\documentclass[nonblindrev]{style/informs3noheader}

\OneAndAHalfSpacedXI
\usepackage[breaklinks, hidelinks, colorlinks = true, linkcolor = blue, urlcolor = blue, anchorcolor = blue, citecolor = blue]{hyperref}
\usepackage[utf8]{inputenc}
\usepackage{graphicx}
\usepackage[numbers]{natbib}
\usepackage{booktabs,caption,fixltx2e,threeparttable}
\usepackage{amssymb}
\usepackage{subcaption}
\usepackage{float}
\usepackage{indentfirst}
\usepackage{textcomp}
\usepackage{mathtools}
\usepackage{tabularx}
\usepackage{amsmath}
\usepackage{amsfonts}
\usepackage{todonotes}
\usepackage{optidef}
\usepackage{bbm}
\usepackage[shortlabels]{enumitem}
\usepackage{algorithm}
\usepackage{algorithmicx}
\usepackage{algpseudocode}
\usepackage{bm}
\usepackage{comment}
\usepackage{multirow}
\usepackage[export]{adjustbox}

\usepackage{tikz}
\usetikzlibrary{matrix,patterns,positioning,decorations,arrows,shapes,decorations.markings,calc,fit}

\newtheorem{defin}{Definition}
\newtheorem{assump}{Assumption}

\newtheorem{prop}{Proposition}

\newtheorem{ex}{Example}

\newcommand{\mbb}{\mathbb}
\newcommand{\mb}{\mathbf}

\usepackage{tikz}
\usepackage{subcaption}
\usetikzlibrary{decorations.pathreplacing, decorations.markings, arrows.meta}

\definecolor{lightBlue}{RGB}{184,207,224}  
\definecolor{medBlue}{RGB}{58,110,165}     
\definecolor{darkBlue}{RGB}{30,63,110}     

\definecolor{bgray50}{RGB}{242,242,240}
\definecolor{bgray100}{RGB}{211,209,199}
\definecolor{bgray200}{RGB}{190,190,185}
\definecolor{bgray400}{RGB}{136,135,128}
\definecolor{bgray600}{RGB}{95,94,90}
\definecolor{bgray800}{RGB}{50,49,47}

\definecolor{alertcolor}{RGB}{30,63,110}

\tikzset{
  mid arrow/.style={
    postaction={decorate,
      decoration={markings,
        mark=at position 0.5 with {\arrow{stealth}}
      }
    }
  },
  mid arrow thick/.style={
    postaction={decorate,
      decoration={markings,
        mark=at position 0.5 with {\arrow[line width=1.5pt]{stealth}}
      }
    }
  },
  bkt/.style      = {rectangle, rounded corners=3pt, inner sep=0pt,
                     fill=white, draw=darkBlue,
                     minimum width=2.0cm, minimum height=2.25cm},
  nd/.style       = {circle, draw=bgray600, fill=bgray50, inner sep=0pt,
                     minimum size=18pt, font=\small},
  block/.style    = {rectangle, rounded corners=6pt},
  el1/.style      = {rectangle, draw=darkBlue, fill=white, inner sep=0pt,
                     line width=0.5pt, minimum size=6pt},
  grarc/.style    = {draw=bgray400, line width=0.5pt, dashed, mid arrow},
  patharc/.style  = {draw=alertcolor, line width=1.5pt, mid arrow thick},
  interarc/.style = {draw=bgray400, line width=0.4pt, mid arrow},
  zoomarc/.style  = {draw=darkBlue, line width=0.5pt, mid arrow},
  brace/.style    = {decorate,
                     decoration={brace, amplitude=5pt, raise=4pt}},
  brace2/.style   = {decorate,
                     decoration={brace, amplitude=5pt, raise=4pt}},
}

\newcommand{\Block}[4]{%
  \node[block, fill=medBlue!20, draw=darkBlue, line width=0.8pt,
        minimum width=2.2cm, minimum height=4.8cm] (#3) at (#1,#2) {};
  \node[font=\small\bfseries, darkBlue, above=2pt]
        at (#3.north) {$#4$};%
}

\newcommand{\Elem}[3]{%
  \node[el1] (#1) at (#2,#3) {};%
}

\newcommand{\Subpath}[2]{%
  \node[circle, draw=darkBlue, fill=white, inner sep=0pt,
        minimum size=7pt] (#1) at #2 {};%
}

\newcommand{\Rep}[2]{%
  \fill[medBlue] #2 circle (3.5pt);%
  \coordinate (#1) at #2;%
}

\newcommand{\SourceSink}[2]{%
  \stnode{s}{(#1, 0)}%
  \stnode{t}{(#2, 0)}%
}

\newcommand{\stnode}[2]{%
  \node[nd] (#1) at #2 {$#1$};%
}

\newcommand{\SubpathBrace}[1]{%
  \draw[brace, darkBlue, thick]
    ([yshift=-6pt]#1.south east) -- ([yshift=-6pt]#1.south west)
    node[midway, below=12pt, font=\small, darkBlue]
    {Subpath};%
}

\usepackage{natbib}
 \bibpunct[, ]{(}{)}{,}{a}{}{,}%
 \def\bibfont{\small}%

\EquationsNumberedThrough
\TheoremsNumberedThrough

\allowdisplaybreaks

\MANUSCRIPTNO{TRSC-0001-2024.00}

\begin{document}

\RUNAUTHOR{Van Rossum, Van Lieshout, and Jacquillat}

\RUNTITLE{Adaptive Partitioning in Column Generation for Nested Paths}
\TITLE{\Large Adaptive Partitioning in Column Generation for Nested Paths}
\ARTICLEAUTHORS{
\AUTHOR{Bart van Rossum}
\AFF{Department of Industrial Engineering \& Innovation Sciences, Eindhoven University of Technology, \EMAIL{b.t.c.v.rossum@tue.nl}}

\AUTHOR{Rolf van Lieshout}
\AFF{Department of Industrial Engineering \& Innovation Sciences, Eindhoven University of Technology, \EMAIL{r.n.v.lieshout@tue.nl}}

\AUTHOR{Alexandre Jacquillat}
\AFF{Sloan School of Management, Massachusetts Institute of Technology, \EMAIL{alexjacq@mit.edu}}
}

\ABSTRACT{
We study a class of nested path problems, in which every path-based variable can be decomposed into a sequence of subpaths. Subpaths must satisfy local resources, while paths must satisfy additional global resources. This paper develops a new exact pricing algorithm in column generation for these problems that avoids the enumeration of non-dominated subpaths. The algorithm relies on adaptive partitioning of subpaths into buckets characterizing the consumption of global path resources. The algorithm represents each bucket by its subpath of minimum reduced cost, and iterates between pessimistic and optimistic pricing steps to combine subpaths into paths while maintaining upper and lower bounds on the minimum reduced cost. An adaptive refinement procedure closes the gap in a finite number of iterations. We demonstrate the effectiveness of the algorithm on two applications. For the balanced multi-period capacitated vehicle routing problem, we obtain speed-ups of up to a factor of 13 over a state-of-the-art column generation benchmark, and the resulting branch-price-and-cut algorithm solves three times as many instances to optimality as a subpath-based baseline. For the robust railway crew scheduling problem, we obtain speed-ups of up to a factor of three and produce primal solutions within 1\% of optimality.
}


\maketitle

\section{Introduction}
\label{sec:intro}

Many transportation planning problems can be modeled as path problems, where a set of elements must be covered by a minimum-cost set of resource-feasible paths. Well-known examples include vehicle routing, where customers are covered by routes, and railway crew scheduling, where tasks are covered by duties. These problems are typically solved with column generation to circumvent the exponential growth in the number of paths \citep{desrosiers2026branch}.

Several problem variants arising in practice feature a \textit{nested path} structure: every column corresponds to a \emph{path}, which can itself be decomposed into a sequence of \emph{subpaths}. Subpaths must satisfy local \emph{subpath resources} pertaining only to their own block of elements, while paths must satisfy additional global \emph{path resources} that tie subpaths together. We present two applications that will serve as motivating examples in this paper and the subject of our computational experiments.

\begin{ex}{Balanced Multi-Period Capacitated Vehicle Routing Problem (B-MPCVRP) \citep{nekooghadirli2026workload}}
\label{ex:cvrp}
We consider a fleet of $K$ homogeneous, capacitated vehicles serving customers over a planning horizon of $T$ days. Each day, a subset of customers $N_t$ must be visited. A \emph{route} is a feasible sequence of customer visits starting and ending at the depot that respects the vehicle capacity. A \emph{schedule} assigns one route to each day of the planning horizon, subject to a maximum total distance $D$ per vehicle. The B-MPCVRP aims to assign a schedule to each vehicle to visit every customer exactly once and minimize routing distance. 
\end{ex}

\begin{ex}{Robust Crew Scheduling Problem (RCSP) \Citep{rahlmann2021robust, rossum2025benders}}
\label{ex:rcsp}
We consider a set of timetable scenarios, each containing a set of tasks (e.g., trips). A \emph{duty} is a feasible sequence of tasks within a scenario that starts and ends at the same depot, and satisfies labor rules on transition times, meal breaks, and maximum duty length. A \emph{template} is a sequence of duties, one per scenario, such that the difference between the latest end time and earliest start time satisfies a maximum template length. The RCSP seeks templates of minimum cost that cover every task at least once in every scenario.
\end{ex}

More broadly, the nested path structure is prevalent across transportation planning problems. For instance, in the electric vehicle routing problem, overall paths can be broken down into subpaths between charging stations \citep{jacquillat2024subpath}. In integrated railway crew replanning, subpaths correspond to crew duties and paths define replanned rosters during scheduled maintenance \citep{breugem2022column}. In integrated airline crew pairing and rostering \citep{saddoune2012integrated}, subpaths define pairings of compatible flights and full paths aggregate them into monthly schedules satisfying rostering constraints.

This paper develops a new exact pricing algorithm for column generation in these nested path problems. A natural baseline is to generate path-based variables directly from individual elements via a labeling algorithm that simultaneously tracks all subpath and path resources. However, this approach fails to exploit the nested structure and may not scale to practically relevant problem sizes as a result. State-of-the-art column generation algorithms for nested path problems therefore further decompose the pricing problem at the subpath level \citep{dohn2013branch, jacquillat2024subpath}. This approach first enumerates all non-dominated subpaths (i.e., those such that no other subpath has both lower reduced cost and lower path resource consumption) and then combines them into paths via a simpler labeling algorithm over a subpath graph. Yet, it can still become computationally intractable when the number of non-dominated subpaths grows large.

In response, we propose an exact pricing algorithm that avoids exhaustive enumeration of non-dominated subpaths via adaptive partitioning. The algorithm partitions subpaths into \emph{buckets} that constrain the consumption of each global path resource. Pricing proceeds on a compact \emph{bucket graph} collecting representative subpaths of minimum reduced cost across buckets, and combining them into full paths subject to global resource constraints. We devise pessimistic and optimistic pricing steps that maintain upper and lower bounds on the minimum reduced cost. A refinement procedure iteratively makes the partition more granular to eliminate the incumbent solution and tighten the bounds. We prove that the adaptive partitioning scheme converges to a path of minimum reduced cost in a finite number of iterations, thus yielding exact column generation and branch-price-and-cut algorithms.

We evaluate the algorithm on two applications with contrasting computational profiles. For the B-MPCVRP, buckets are defined along a single resource dimension but finding bucket representatives is computationally demanding. The adaptive partitioning algorithm achieves speed-ups of up to a factor of 13 over a column generation benchmark relying on the enumeration of non-dominated subpaths. These benefits enable the implementation of a path-based branch-price-and-cut algorithm, which solves three times as many instances to optimality as a subpath-based branch-price-and-cut baseline. For the RCSP, buckets are defined along two resource dimensions and finding representatives is cheap, but convergence is slowed by the degenerate problem structure. We obtain speed-ups of up to a factor of three, and use a diving heuristic to find solutions within 1\% of optimality. Due to the contrasting nature of these two problems, the adaptive partitioning algorithm provides benefits through different mechanisms and with different configurations. Specifically, when the pricing problem is already hard over subpaths (as in the B-MPCVRP), improvements primarily stem from avoiding the enumeration of all non-dominated subpaths. In contrast, when the pricing problem is easier over subpaths (as in the RCSP), the improvements stem from maintaining a coarser bucket graph when combining subpaths into paths.

In summary, this paper makes three main contributions. First, we define a class of nested path problems, encompassing a broad range of problems arising in transportation planning and operations. Second, we propose an adaptive partitioning algorithm in column generation for these problems without necessitating ex-ante subpath enumeration. The algorithm provably solves the pricing problem to optimality over a sparser ``bucket graph'' in a finite number of iterations, which gets refined over time to tighten lower and upper bounds on the minimum reduced cost. Third, we demonstrate the computational efficiency of the algorithm on two transportation applications. Results identify substantial speed-ups over state-of-the-art column generation benchmarks, yielding stronger solutions and tighter bounds in large-scale and otherwise intractable problems.

The remainder of this paper is structured as follows. We review related literature in Section~\ref{sec:literature}. In Section~\ref{sec:problem}, we introduce the general nested path problem and describe the state-of-the-art benchmark pricing algorithm. In Section~\ref{sec:bucketPricing}, we present the adaptive partitioning algorithm and establish its correctness and finite convergence. In Sections~\ref{sec:cvrp} and~\ref{sec:crew}, we demonstrate the algorithm's effectiveness on the balanced multi-period vehicle routing problem and the robust railway crew scheduling problem, respectively. In Section~\ref{sec:conclusion}, we conclude and outline future research directions.

\section{Literature Review}
\label{sec:literature}
Our work relates to decomposition algorithms for nested path problems in transportation and to solution methods based on adaptive discretization. We note that the nested path structure studied in this paper is distinct from the literature on `nested column generation', in which the pricing problem is itself solved via column generation \citep{vanderbeck2001nested}.

\paragraph{Nested paths.}
Nested path problems can be modeled via path-based or subpath-based formulations. A subpath-based formulation quells the exponential growth of the number of variables, but requires explicit linking constraints to model the path resources, which might result in a weaker linear relaxation and higher symmetry. Column generation methods on these formulations involve adding subpath-based variables iteratively \citep{sadykov2013column,delorme2020enhanced,jacquillat2022optimizing,martin2024double}. \citet{breugem2022column} employ a duty-based formulation in the integrated crew rescheduling problem with explicit rostering constraints, using valid inequalities to tighten the linear relaxation. \Citet{rossum2025benders} propose a duty-based formulation for the RCSP and employ Benders decomposition to accommodate side constraints on admissible combinations of templates. A path-based formulation, by contrast, tightens the linear relaxation by relegating linking constraints to the pricing problem \citep{desrosiers2026branch,yamin2026dantzig}. However, this comes at the cost of increased complexity in the pricing problem to find a path (rather than a subpath) of minimum reduced cost. In response, this paper proposes an exact and efficient pricing algorithm for path-based formulations of nested path problems.

A possible pricing algorithm for path-based formulations involves generating full paths directly. For instance, \citet{saddoune2012integrated} price out entire crew schedules from flights, bypassing the intermediate level of flight pairings. However, this approach can become computationally intractable as paths become longer. This limitation motivates alternative two-stage approaches exploiting the nested path structure in the pricing problem. In the RCSP (Example~\ref{ex:rcsp}), \citet{rahlmann2021robust} first aggregate tasks into duties and then duties into templates. They modify the reduced cost criterion to avoid full enumeration of duties, but do not establish convergence guarantees. Similarly, in the B-MPCVRP (Example~\ref{ex:cvrp}), \citet{nekooghadirli2026workload} propose a sequential heuristic that first determines cost-efficient routes per day and subsequently assigns them to schedules to minimize the maximum workload across vehicles. More closely related to our work, \citet{dohn2013branch} study a generalized staff rostering problem in which shifts are assigned to staff members in the form of rosters, and solve it via column generation with a three-stage procedure that enumerates all non-dominated roster lines. In the electric vehicle routing problem, \citet{jacquillat2024subpath} develop a two-stage pricing algorithm enumerating non-dominated subpaths and then combining them into full paths, with theoretical guarantees of correctness and finite convergence.

This paper contributes a new pricing algorithm that exploits the nested structure of the problem, but does not rely on the explicit enumeration of all non-dominated subpaths. Instead, the adaptive partitioning scheme maintains a coarsened approximation of the subpath space in a path-based labeling scheme, and iteratively refines it until convergence to a path of minimum reduced cost.

\paragraph{Adaptive discretization.}
Our algorithm falls into a broader stream of literature that maintains a relaxation of the problem of interest and iteratively refines it based on the incumbent solution. In a continuous service network design problem, the dynamic discretization discovery approach from \citet{boland2017continuous} relies on a partially time-expanded network with a coarse enough discretization to induce a relaxation, and refines the discretization iteratively when the incumbent solution violates time-window requirements. In a capacitated vehicle routing problem, the column elimination algorithm from \citet{karahalios2023column} represents routes as integer flows through a compact decision diagram, and iteratively removes violated paths. In a facility location setting with demand-supply interactions, \citet{kai2022vertiport} develop an exact adaptive discretization algorithm that approximates a nonconvex demand function with piecewise linear segments, iterating between a conservative and a relaxed model. Rather than maintaining a relaxation of the master problem, our algorithm refines a coarsened approximation of the subpath graph in the pricing problem.

This distinction relates to iterative pricing algorithms in column generation. \citet{righini2008new} and \citet{boland2006accelerated} develop a decremental state space relaxation approach that solves a resource-constrained elementary shortest path problem over a relaxed state space and dynamically adds additional resources whenever the resulting paths violate the original constraints. Related to our bucket-based partitioning scheme, \Citet{jaarsveld2015improving} solve a spare part inventory control problem via column generation in which the pricing problem seeks single-item $(s, S)$ policies by iteratively refining a grid of parallelograms. Our  adaptive partitioning algorithm develops, for the first time to our knowledge, this methodological idea for the nested path structure. By partitioning the subpath space into buckets based on resource consumption, the algorithm exploits the separation between subpath and path resources to decompose the pricing problem across subpaths while maintaining guarantees of exactness and finite convergence.

\section{Problem Description}
\label{sec:problem}

Consider a directed graph $G = (\{s, t\} \cup K, A)$, consisting of a source $s$, a sink $t$, and a set $K$ of elements connected through a set $A$ of arcs. Each arc $a \in A$ has cost $c_a$. We define a set $R$ of resource constraints. The problem consists of covering all elements in $K$ by a minimum-cost set of source-sink paths that satisfy the resource constraints stored in $R$.

We formalize a nested path structure by making assumptions on the graph $G$ and the set of resource constraints $R$. First, the set $K$ can be partitioned into $K_1, \dots, K_n$ ordered subsets, called \emph{blocks}. Elements within a block may be connected, but inter-block arcs can only pass from block $K_i$ to $K_{i+1}$. We assume full connectivity between consecutive blocks, i.e., the inter-block arc set is $K_i \times K_{i+1}$ for each $i = 1, \dots, n-1$. Similarly, every element in $K_1$ is reachable from the source $s$, and every element in $K_n$ is connected to the sink $t$. As such, any path $\pi$ in $G$ can be uniquely written as a sequence of subpaths $\pi = (s,\sigma_1, \dots, \sigma_n,t)$, where subpath $\sigma_i$ only traverses elements in block $K_i$. Figure~\ref{fig:element_graph} illustrates an element graph consisting of three blocks.

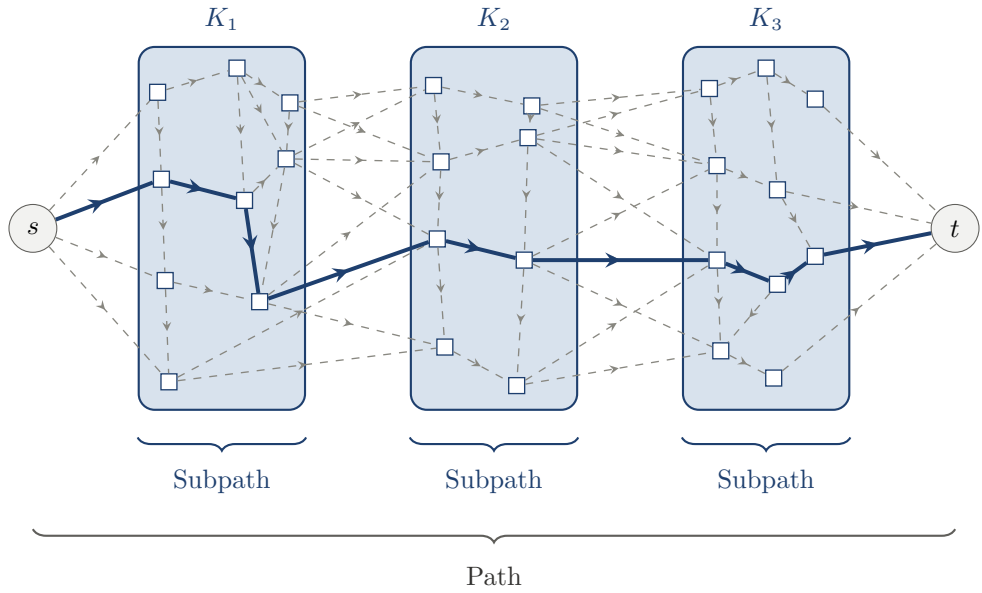
\begin{figure}[h!]
\centering
\begin{tikzpicture}

  \SourceSink{-2.5}{9.7}

  \Block{0}{0}{K1}{K_1}
  \Elem{a1}{-0.85}{ 1.80}  \Elem{a2}{ 0.20}{ 2.12}  \Elem{a3}{ 0.90}{ 1.66}
  \Elem{a4}{-0.80}{ 0.65}  \Elem{a5}{ 0.30}{ 0.37}  \Elem{a6}{ 0.85}{ 0.92}
  \Elem{a7}{-0.75}{-0.69}  \Elem{a8}{ 0.50}{-0.97}  \Elem{a9}{-0.70}{-2.03}

  \draw[grarc] (a1)--(a2); \draw[grarc] (a2)--(a3);
  \draw[grarc] (a1)--(a4); \draw[grarc] (a2)--(a5); \draw[grarc] (a2)--(a6);
  \draw[grarc] (a3)--(a6); \draw[grarc] (a4)--(a5); \draw[grarc] (a5)--(a6);
  \draw[grarc] (a4)--(a7); \draw[grarc] (a5)--(a8); \draw[grarc] (a6)--(a8);
  \draw[grarc] (a7)--(a8); \draw[grarc] (a7)--(a9);
  \draw[grarc] (s)--(a1);  \draw[grarc] (s)--(a4);
  \draw[grarc] (s)--(a7);  \draw[grarc] (s)--(a9);

  \Block{3.6}{0}{K2}{K_2}
  \Elem{b1}{2.80}{ 1.89}  \Elem{b2}{4.10}{ 1.62}
  \Elem{b3}{2.90}{ 0.88}  \Elem{b4}{4.05}{ 1.20}
  \Elem{b5}{2.85}{-0.14}  \Elem{b6}{4.00}{-0.42}
  \Elem{b7}{2.95}{-1.57}  \Elem{b8}{3.90}{-2.08}

  \draw[grarc] (b1)--(b2); \draw[grarc] (b1)--(b3); \draw[grarc] (b2)--(b4);
  \draw[grarc] (b3)--(b4); \draw[grarc] (b3)--(b5); \draw[grarc] (b4)--(b6);
  \draw[grarc] (b5)--(b6); \draw[grarc] (b5)--(b7); \draw[grarc] (b6)--(b8);
  \draw[grarc] (b7)--(b8);
  \draw[grarc] (a3)--(b1); \draw[grarc] (a3)--(b3);
  \draw[grarc] (a6)--(b1); \draw[grarc] (a6)--(b3); \draw[grarc] (a6)--(b5);
  \draw[grarc] (a8)--(b3); \draw[grarc] (a8)--(b5); \draw[grarc] (a8)--(b7);
  \draw[grarc] (a9)--(b5); \draw[grarc] (a9)--(b7);

  \Block{7.2}{0}{K3}{K_3}
  \Elem{c1}{6.45}{ 1.85}  \Elem{c2}{7.20}{ 2.12}  \Elem{c3}{7.85}{ 1.71}
  \Elem{c4}{6.55}{ 0.83}  \Elem{c5}{7.35}{ 0.51}
  \Elem{c6}{6.55}{-0.42}  \Elem{c7}{7.35}{-0.74}  \Elem{c8}{7.85}{-0.37}
  \Elem{c9}{6.60}{-1.62}  \Elem{c10}{7.30}{-1.98}

  \draw[grarc] (c1)--(c2); \draw[grarc] (c2)--(c3); \draw[grarc] (c1)--(c4);
  \draw[grarc] (c4)--(c5); \draw[grarc] (c2)--(c5); \draw[grarc] (c4)--(c6);
  \draw[grarc] (c6)--(c7); \draw[grarc] (c7)--(c8); \draw[grarc] (c5)--(c8);
  \draw[grarc] (c6)--(c9); \draw[grarc] (c7)--(c9); \draw[grarc] (c9)--(c10);
  \draw[grarc] (b2)--(c1); \draw[grarc] (b2)--(c4);
  \draw[grarc] (b4)--(c1); \draw[grarc] (b4)--(c4); \draw[grarc] (b4)--(c6);
  \draw[grarc] (b6)--(c4); \draw[grarc] (b6)--(c6); \draw[grarc] (b6)--(c9);
  \draw[grarc] (b8)--(c6); \draw[grarc] (b8)--(c9);
  \draw[grarc] (c3)--(t);  \draw[grarc] (c5)--(t);
  \draw[grarc] (c8)--(t);  \draw[grarc] (c10)--(t);

  \draw[patharc] (s)--(a4);  \draw[patharc] (a4)--(a5); \draw[patharc] (a5)--(a8);
  \draw[patharc] (a8)--(b5); \draw[patharc] (b5)--(b6);
  \draw[patharc] (b6)--(c6); \draw[patharc] (c6)--(c7);
  \draw[patharc] (c7)--(c8); \draw[patharc] (c8)--(t);

  \SubpathBrace{K1}
  \SubpathBrace{K2}  
  \SubpathBrace{K3}

  \draw[brace2, bgray600, thick]
  ([yshift=-100pt]t.south) -- ([yshift=-100pt]s.south)
  node[midway, below=14pt, font=\small, bgray800]
  {Path};

\end{tikzpicture}
\caption{Element graph consisting of three blocks.}
\label{fig:element_graph}
\end{figure}

Second, the resources $R$ can be partitioned into path and subpath resources accordingly. In particular, we assume that $R = \bigcup\limits_{i=0}^{n} R_i$, where $R_0$ denotes the set of global path resources, applied to all arcs in $A$, and $R_1,\dots,R_n$ store local subpath resources applied to arcs within block $K_i$. We further formalize this notion in Section~\ref{subsec:problem_resource}.

This nested path structure is satisfied in our two motivating examples. In the B-MPCVRP (Example~\ref{ex:cvrp}), elements are customers, blocks correspond to days, and each block contains the customers to be visited on that day. A subpath is a daily vehicle route satisfying capacity and routing constraints, and a path is a multi-day schedule subject to a maximum total distance. In the RCSP (Example~\ref{ex:rcsp}), elements are tasks, and each block corresponds to a timetable scenario. A subpath is a duty covering tasks within a single scenario subject to labor rules, while a path is a template coupling one duty per scenario subject to the maximum template length constraint.

We now provide a general path-based mathematical formulation for the problem. Let $\mathcal{P}$ denote the set of resource-feasible paths in $G$. For each path $\pi \in \mathcal{P}$, let $c_\pi$ denote its cost and $a_{\pi k}\in\{0,1\}$ indicate whether element $k \in K$ is covered by $\pi$. We define the binary decision variable $x_\pi$ encoding whether path $\pi$ is selected. A set covering formulation of the path problem is given by:
\begin{subequations}
\begin{align}
\min \quad & \sum_{\pi \in \mathcal{P}} c_\pi x_\pi & \\ 
\text{s.t.} \quad & \sum_{\pi \in \mathcal{P}} a_{\pi k} x_\pi \geq 1 
    && \forall k \in K \label{eq:master_cover} \\ 
& x_{\pi} \in \{0, 1\} && \forall \pi \in \mathcal{P}.
\end{align}
\label{eq:master_ip}%
\end{subequations}
Without loss of generality, the set covering constraints can be replaced by partitioning constraints, and the model can be augmented with side constraints. 

\subsection{Resource Constraints}
\label{subsec:problem_resource}

We now formalize the notion of resource feasibility, referring to \citet{irnich2008resource} for a detailed introduction to resource constraints and resource extension functions. For ease of exposition, we first present the standard case in which every resource is one-dimensional and admissible values form an interval, and generalize it to multi-dimensional resources thereafter.

\subsubsection*{Single-dimensional resource constraints.}
Each resource $r\in R$ takes integer values, and tracks a quantity that accumulates along a path subject to bounds at each node. Typical examples include vehicle load, battery charge, or elapsed time. The resource consumption along a path is governed through a resource extension function $f_{a}(\cdot): \mbb{Z}^{R} \rightarrow \mbb{Z}^{R}$ for each arc $a\in A$, which specifies how a resource value progresses along the arc. Every node $k \in \{s, t\} \cup K$ specifies an interval $[\ell_k^r,u_k^r]$ for each resource $r$. Without loss of generality, we assume that $\ell_s^r=u_s^r=0$ for all $r$.

Any path $\pi = (s = v_0, v_1, \dots, v_{p-1}, v_p = t)$ in $G$ traversing arcs $a_1, \dots, a_{p}$ gives rise to a sequence of resource vectors $(\mathbf{T}_{0}, \mathbf{T}_{1}, \dots, \mathbf{T}_{p})$, which we call the resource trajectory of $\pi$. This trajectory is generated by applying the resource extension functions sequentially along the arcs of $\pi$:
\begin{equation}
    \mathbf{T}_{0} = \mathbf{0}, \qquad \mathbf{T}_{m} = f_{a_m}(\mathbf{T}_{m-1}) \quad \text{for } m = 1, \dots, p.
\end{equation}
Writing $T^r_m$ for the $r$-th entry of $\mathbf{T}_m$, we say that $\pi$ is \emph{$R$-feasible} when $T^r_m \in [\ell^r_{v_m}, u^r_{v_m}]$ for all $r \in R$ and all $m = 0, \dots, p$. Let $\mathcal{P}(G, R)$ denote the set of $R$-feasible paths in $G$, and write $\mathcal{P} = \mathcal{P}(G, R)$ for short. We are now ready to formalize the notion of subpath resources.

\begin{defin}[Subpath resources]
\label{def:subpath_resource}
Resource $r \in R$ is a subpath resource of block $i$ if (i) $[\ell^r_k, u^r_k] = (-\infty, \infty)$ for all $k \notin \{s\} \cup K_i$ and (ii) $(f_{a}(\mathbf{T}))^r = T^{r}$ for all $a \notin A_i$.
\end{defin}

Condition~(i) states that subpath resources impose no constraints outside their own block, and Condition~(ii) specifies that their consumption is unaffected by arcs outside that block.

As stated before, the resource set $R$ can be partitioned into subpath resources and path resources. In particular, we write $R = \bigcup\limits_{i=0}^{n} R_i$, where $R_0$ is the set of path resources and $R_i$ is the set of subpath resources for block $K_i$. A resource is a path resource when it is not a subpath resource.

This partitioning induces a corresponding decomposition of feasible paths. By Definition~\ref{def:subpath_resource}, any feasible path can be decomposed by block: once a subpath satisfies the subpath resources of its block, the subpath constraints can be fully discarded when concatenating subpaths into a path. We denote by $G_i = (K_i, A_i)$ the subgraph induced by block $K_i$ and write $\mathcal{S}_i$ for the set of feasible subpaths in $G_i$ with respect to $R_i$. We assume that $\mathcal{S}_i$ is finite for all $i = 1, \dots, n$.

For any subpath $\sigma = (k_0, k_1, \dots, k_q) \in \mathcal{S}_i$ traversing arcs $a_1, \dots, a_q$, we define the resource trajectory of $\sigma$ recursively by
\begin{equation}
    \mathbf{T}_{0} = \mathbf{0}, \qquad \mathbf{T}_{m} = f_{a_m}(\mathbf{T}_{m-1}) \quad \text{for } m = 1, \dots, q.
\end{equation}
We call $t^{r}_{\sigma} = T^{r}_q$ the \emph{contribution} of $\sigma$ to resource $r\in R$. Collecting the contributions to all path resources, we write $\mathbf{t}^{R_0}_{\sigma} = \{t^{r}_{\sigma}:r \in R_0\} \in \mbb{Z}^{R_0}$ for the \emph{path-resource consumption vector} of $\sigma$.

We impose two structural properties on the path resources that our algorithm will exploit.

\begin{assump}[Path resources]
\label{assump:path_resources}
Every path resource $r \in R_0$ satisfies the following properties:

\noindent (i) \emph{Downward closure.} For $k \in \{s, t\} \cup K$, the admissible interval is unbounded below: $\ell^r_k=-\infty$.

\noindent (ii) \emph{Monotone separable aggregation.} For every block $i = 1, \dots, n$, there exists an aggregation function $F^{r}_{i}: \mbb{Z}^{i} \rightarrow \mbb{Z}$, non-decreasing in each argument, such that the value of resource $r$ at the end of block $i$ along any path $\pi = (s, \sigma_1, \dots, \sigma_n, t)$ equals $F^{r}_{i}(t^{r}_{\sigma_1}, \dots, t^{r}_{\sigma_i})$.
\end{assump}

It follows that reducing the contribution of a subpath to a resource $r \in R_0$ can never render a feasible path infeasible; equivalently, underestimating the resource consumption of each subpath yields a valid relaxation of the pricing problem. These properties are crucial for establishing valid upper and lower bounds on the minimum reduced cost throughout our pricing algorithm. This assumption is also commonly satisfied in practical applications. In the B-MPCVRP for instance, the sole path resource tracks cumulative route distance, which is non-decreasing across days, and the distance restriction merely imposes an upper bound on the total distance.

\subsubsection*{Multi-dimensional resource constraints.}

To capture more general applications, we allow resources to take values in $\mbb{Z}^{d_r}$ with $d_r > 1$ with broader restrictions than intervals. For example, in the RCSP, the template length is governed by the difference between the latest end time and the earliest start time across duties, which requires a two-dimensional representation of path-based resources to track both endpoints. The corresponding set of admissible values is defined by a single linear inequality and is not a Cartesian product of intervals, as illustrated in Figure~\ref{fig:rcsp_feasible_set}. To accommodate such cases, we now extend the framework above to multi-dimensional resources.


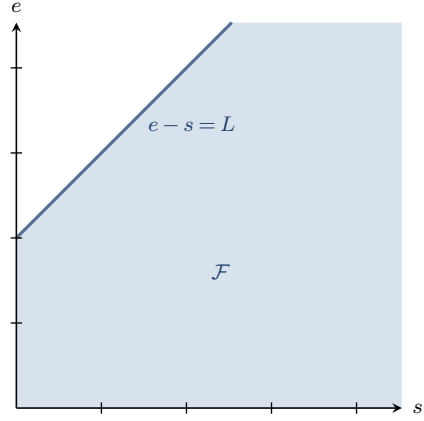
\begin{figure}[htbp!]
\centering
\begin{tikzpicture}[scale=1.5]

  \fill[medBlue!20]
    (0, 0) -- (3.40, 0) -- (3.40, 3.40) -- (1.90, 3.40) -- (0, 1.50) -- cycle;

  \draw[darkBlue!75, line width=1.2pt]
    (0, 1.50) -- (1.90, 3.40);
  \node[font=\scriptsize\sffamily, darkBlue, right=3pt] at (1.00, 2.50) {$e - s = L$};

  \draw[->, >=stealth, black, line width=0.6pt]
    (0, 0) -- (3.40, 0)
    node[right, font=\scriptsize\sffamily] {$s$};
  \draw[->, >=stealth, black, line width=0.6pt]
    (0, 0) -- (0, 3.40)
    node[above, font=\scriptsize\sffamily] {$e$};

  \foreach \x in {0.75, 1.50, 2.25, 3.0} {
    \draw[line width=0.5pt] (\x, 0.05) -- (\x, -0.05);
  }

  \foreach \y in {0.75, 1.50, 2.25, 3.0} {
    \draw[line width=0.5pt] (0.05, \y) -- (-0.05, \y);
  }

  \node[font=\scriptsize\sffamily, darkBlue] at (1.80, 1.20) {$\mathcal{F}$};

\end{tikzpicture}
\caption{Feasible set $\mathcal{F} = \{(s_\sigma, e_\sigma) : e_\sigma - s_\sigma \leq L\}$ of duty contribution vectors for the RCSP path resource. The admissible region (shaded) is bounded above by the line $e_\sigma - s_\sigma = L$.}
\label{fig:rcsp_feasible_set}
\end{figure}

Each resource $r \in R$ has dimension $d_r \in \mbb{Z}_{\geq 1}$ and takes values in $\mbb{Z}^{d_r}$. The joint resource space becomes $\mbb{Z}^{R} = \prod_{r \in R} \mbb{Z}^{d_r}$, endowed with the componentwise order~$\preceq$. The resource extension functions and resource trajectories are defined exactly as before. Since the per-resource intervals $[\ell^r_k, u^r_k]$ no longer suffice to describe admissibility, we associate with each node $k \in \{s, t\} \cup K$ an \emph{admissible set} $\mathcal{F}_k \subseteq \mbb{Z}^{R}$, and say that a path is $R$-feasible when $\mathbf{T}_m \in \mathcal{F}_{v_m}$ for all $m$. We define $\mathbf{F}^{r}_k \subseteq \mbb{Z}^{d_r}$ as its projection onto resource $r$. The notion of subpath resources extends accordingly, with condition~(i) of Definition~\ref{def:subpath_resource} modified as $\mathbf{F}^{r}_k = \mbb{Z}^{d_r}$. The contribution vector $\mathbf{t}^{r}_{\sigma}$ now lies in $\mbb{Z}^{d_r}$.

The two structural properties of Assumption~\ref{assump:path_resources} generalize naturally as follows. Downward closure now requires the admissible projection $\mathbf{F}^{r}_k$ to be downward-closed with respect to the componentwise order, meaning that $\mathbf{t}' \in \mathbf{F}^{r}_k$ whenever $\mathbf{t}' \preceq \mathbf{T}$ for some $\mathbf{T} \in \mathbf{F}^{r}_k$. The aggregation function $F^{r}_{i}: (\mbb{Z}^{d_r})^{i} \rightarrow \mbb{Z}^{d_r}$ is required to be componentwise non-decreasing in each argument.

We verify that the RCSP (Example~\ref{ex:rcsp}) satisfies the generalized assumption. Each duty contributes a pair consisting of its negated start time and its end time. The aggregation across a sequence of duties takes the componentwise maximum, and both components are non-decreasing in each duty's contribution. The admissible set is defined by the template length constraint $\text{end} - \text{start} \leq L$, which is downward-closed with respect to the componentwise order: reducing any entry of a duty's contribution vector, i.e., starting later or ending earlier, can only decrease the template length.

\subsection{Column Generation Benchmark}
\label{subsec:problem_benchmark}

\paragraph{Preliminaries.}
Due to the large number of path variables, column generation is the method of choice for solving model~\eqref{eq:master_ip}. The restricted master problem (RMP) is defined as the problem's linear relaxation over a subset of paths $\mathcal{P}'\subseteq\mathcal{P}$:
\begin{subequations}
\begin{align}
\min \quad & \sum_{\pi \in \mathcal{P}'} c_\pi x_\pi & \\ 
\text{s.t.} \quad & \sum_{\pi \in \mathcal{P}'} a_{\pi k} x_\pi \geq 1 && \forall k \in K \label{eq:rmp_cover} \\ 
& x_{\pi} \geq 0 && \forall \pi \in \mathcal{P}'.
\end{align}
\label{eq:rmp}%
\end{subequations}
In each column generation iteration, upon solving the RMP, the pricing problem seeks a variable of negative reduced cost or a certificate that none exists. In the former case, the algorithm repeats. In the latter case, the incumbent solution to the RMP is provably optimal for the linear programming (LP) relaxation of model~\eqref{eq:master_ip}. An integral solution can be obtained by embedding column generation into a branch-price-and-cut algorithm or via primal heuristics \citep{desrosiers2026branch}.

Denoting by $\mathbf{\lambda}$ the dual variables of constraints~\eqref{eq:rmp_cover}, the reduced cost of path $\pi \in \mathcal{P}$ reads 
\begin{align}
    \Tilde{c}_{\pi} = c_\pi - \sum_{k \in \pi} \lambda_k. \label{eq:reduced_cost_path}%
\end{align} 
Since path costs are sum of arc costs, the reduced cost of a path can be fully decomposed over its arcs. The reduced cost of a subpath is defined analogously.

\paragraph{Path-based benchmark.}
A natural benchmark in the pricing problem involves generating a path-based variable by applying a labeling algorithm directly on the element graph $G$ and simultaneously tracking subpath and path resources along path extensions. This approach fails to exploit the nested structure of the problem: labels must track the full resource vector at every node, and the resulting state space grows with the combined dimension of all subpath and path resources. 

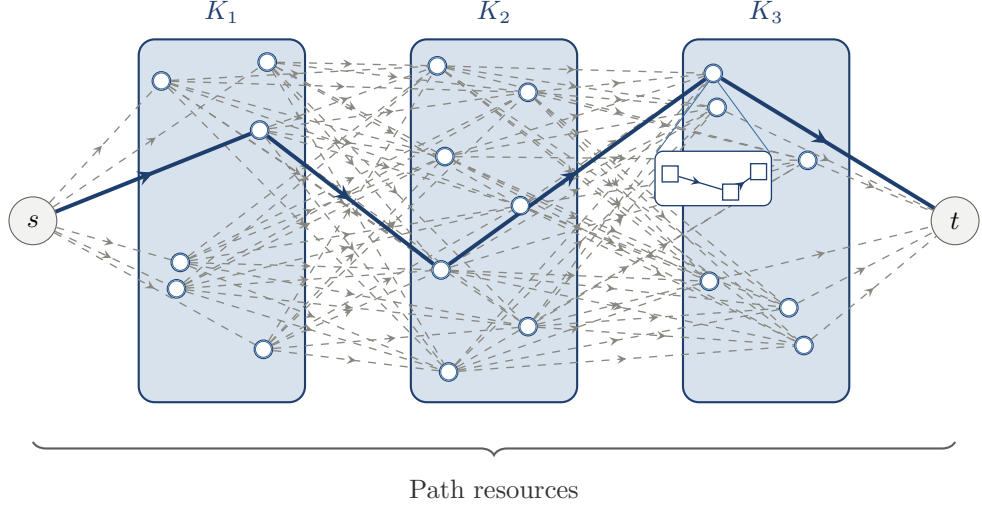
\begin{figure}[htbp]
\centering
\begin{tikzpicture}

  \SourceSink{-2.5}{9.7}

  \Block{0}{0}{K1}{K_1}
  \Subpath{p1}{(-0.80, 1.85)}
  \Subpath{p2}{( 0.50, 1.20)}
  \Subpath{p3}{( 0.60, 2.10)}
  \Subpath{p4}{(-0.55,-0.55)}
  \Subpath{p5}{( 0.55,-1.70)}
  \Subpath{p6}{(-0.60,-0.90)}

  \Block{3.6}{0}{K2}{K_2}
  \Subpath{q1}{(2.85, 2.05)}
  \Subpath{q2}{(4.05, 1.70)}
  \Subpath{q3}{(2.95, 0.85)}
  \Subpath{q4}{(3.95, 0.20)}
  \Subpath{q5}{(2.90,-0.65)}
  \Subpath{q6}{(4.05,-1.40)}
  \Subpath{q7}{(3.00,-2.00)}

  \Block{7.2}{0}{K3}{K_3}
  \Subpath{r1}{(6.50, 1.95)}
  \Subpath{r2}{(7.75, 0.80)}
  \Subpath{r3}{(6.45,-0.80)}
  \Subpath{r4}{(7.70,-1.65)}
  \Subpath{r5}{(6.55, 1.50)}
  \Subpath{r6}{(7.50,-1.15)}

  \foreach \i in {p1,p2,p3,p4,p5,p6} {
    \draw[grarc] (s) -- (\i); }

  \foreach \i in {p1,p2,p3,p4,p5,p6} {
    \foreach \j in {q1,q2,q3,q4,q5,q6,q7} {
      \draw[grarc] (\i) -- (\j); } }

  \foreach \i in {q1,q2,q3,q4,q5,q6,q7} {
    \foreach \j in {r1,r2,r3,r4,r5,r6} {
      \draw[grarc] (\i) -- (\j); } }

  \foreach \i in {r1,r2,r3,r4,r5,r6} {
    \draw[grarc] (\i) -- (t); }

  \draw[patharc] (s)  -- (p2);
  \draw[patharc] (p2) -- (q5);
  \draw[patharc] (q5) -- (r1);
  \draw[patharc] (r1) -- (t);

  \draw[medBlue, line width=0.4pt] (6.50, 1.92) -- (5.75, 0.90);
  \draw[medBlue, line width=0.4pt] (6.50, 1.92) -- (7.25, 0.90);
  \draw[fill=white, draw=darkBlue, rounded corners=3pt, line width=0.5pt]
    (5.73, 0.20) rectangle (7.27, 0.92);
  \Elem{z1}{5.93}{0.62}   
  \Elem{z2}{6.73}{0.38}   
  \Elem{z3}{7.10}{0.65}   
  \draw[zoomarc] (z1) -- (z2);
  \draw[zoomarc] (z2) -- (z3);

  \foreach \i in {p1,p2,p3,p4,p5,p6,q1,q2,q3,q4,q5,q6,q7,r1,r2,r3,r4,r5,r6} {
    \node[circle, draw=medBlue, fill=white, inner sep=0pt,
          minimum size=6pt] at (\i) {};
  }

\draw[brace2, bgray600, thick]
  ([yshift=-70pt]t.south) -- ([yshift=-70pt]s.south)
  node[midway, below=14pt, font=\small, bgray800]
  {Path resources};

\end{tikzpicture}
\caption{Subpath graph corresponding to the element graph of Figure~\ref{fig:element_graph}. The zoom box illustrates that every node in this graph corresponds to a subpath in the element graph.}
\label{fig:subpath_graph}
\end{figure}

\paragraph{Enumerative benchmark.}
We define a state-of-the-art benchmark based on the two-stage decomposition from \citet{dohn2013branch, jacquillat2024subpath}. Rather than generating paths from individual elements, this approach instead exploits the nested problem structure by applying a labeling algorithm over a \emph{subpath graph} $H = (\{s, t\} \cup \mathcal{S}, E)$, illustrated in Figure~\ref{fig:subpath_graph}. In this graph, each node corresponds to a feasible subpath $\sigma \in \mathcal{S} = \bigcup_{i=1}^{n} \mathcal{S}_i$. By full connectivity, an arc $(\sigma, \sigma') \in E$ exists whenever $\sigma \in \mathcal{S}_i$ and $\sigma' \in \mathcal{S}_{i+1}$ for some $i$, i.e., any two subpaths in consecutive blocks can be concatenated. Every resource-feasible path in $G$ corresponds to a path in $H$, and vice versa. The subpath graph thus provides an equivalent representation of the problem. 

The enumerative benchmark proceeds in two steps. First, it enumerates all non-dominated subpaths in $\mathcal{S}_i$, for each block $i = 1, \dots, n$. A subpath $\sigma \in \mathcal{S}_i$ is dominated when there exists another subpath $\sigma' \in \mathcal{S}_i$ with lesser or equal reduced cost and lesser or equal consumption of every path resource in $R_0$, with at least one inequality strict. Second, it applies a labeling algorithm on the subgraph of $H$ induced by all non-dominated subpaths, where only path resources $R_0$ need be tracked. Whereas this approach decomposes the pricing problem across subpaths, it can still be computationally intensive in large-scale instances with weak domination patterns across subpaths---leading to long enumeration times in the first step and long labeling times in the second step. We present next an adaptive partitioning pricing algorithm to address these issues, and compare it to the enumerative benchmark in Sections~\ref{sec:cvrp} and~\ref{sec:crew}.

\section{Adaptive Partitioning Algorithm}
\label{sec:bucketPricing}

We now present an adaptive bucket-based partitioning algorithm, which solves the pricing problem by working with a coarsened representation of the subpath graph $H$. We present the algorithm in full generality, using the multi-dimensional notation from Section~\ref{subsec:problem_resource}.

The algorithm relies on two key concepts. First, we partition the subpaths in each block into \emph{buckets} defined by resource intervals. Each bucket is represented by its minimum reduced cost subpath, called the \emph{representative} for short. We construct a coarse \emph{bucket graph} connecting all representatives toward forming full paths. Pricing proceeds on this bucket graph by iterating between pessimistic and optimistic pricing steps, which together maintain valid upper and lower bounds on the minimum reduced cost. Since the bucket graph is a coarse representation of the full subpath graph, pricing on the bucket graph remains more tractable than pricing directly on $H$. Second, an adaptive bucket refinement approach ensures that these bounds converge within a finite number of iterations. Refinement decisions are guided by the current dual information to split buckets along ``promising'' paths based on the incumbent dual variables. As a result, the algorithm focuses its efforts on regions of the subpath space that are most relevant for the current RMP solution. Figure~\ref{fig:flowchart} provides an overview of the algorithm.

\begin{figure}[htbp!]
\centering

\begin{tikzpicture}[node distance=0.5cm,
  terminal/.style={rectangle, rounded corners=4pt,
    minimum width=2.5cm, align=center,
    draw=bgray600, fill=bgray50,
    text=bgray800,
    font=\small\bfseries,
    line width=0.5pt},
  process/.style={rectangle, rounded corners=4pt,
    minimum width=2.5cm, align=center,
    draw=darkBlue, fill=medBlue!20,
    text=darkBlue,
    font=\small,
    line width=0.5pt},
  arrow/.style={->, >=stealth, bgray600, line width=0.6pt},
  lbl/.style={font=\scriptsize, bgray600},
]
  \node (init)      [terminal]                              {Bucket partition $\mathcal{B}$};
  \node (solvermp)  [process,  below=of init]               {Solve RMP};
  \node (rep)       [process,  below=of solvermp]           {Compute bucket\\ representatives};
  \node (pess)      [process,  below=of rep]                {Pessimistic pricing\\ $\tilde{c}^{*}_{pes}(\mathcal{B}) < 0$?};
  \node (addrmp)    [process,  left=of pess, xshift=-1.2cm] {Add path\\ to RMP};
  \node (opt)       [process,  below=of pess]               {Optimistic pricing\\ $\tilde{c}^{*}_{opt}(\mathcal{B}) < 0$?};
  \node (refine)    [process,  right=of opt,  xshift=1.2cm] {Refine $\mathcal{B}$};
  \node (terminate) [terminal, below=of opt]                {RMP is optimal};

  \draw[arrow] (init)         -- (solvermp);
  \draw[arrow] (solvermp)     -- node[lbl, anchor=east]  {Duals $\mb{\lambda}$} (rep);
  \draw[arrow] (rep)          -- (pess);
  \draw[arrow] (pess.west)    -- node[lbl, anchor=south] {Yes} (addrmp.east);
  \draw[arrow] (addrmp.north) |- (solvermp.west);
  \draw[arrow] (pess.south)   -- node[lbl, anchor=east]  {No}  (opt.north);
  \draw[arrow] (opt.east)     -- node[lbl, anchor=south] {Yes} (refine.west);
  \draw[arrow] (refine.north) |- (rep.east);
  \draw[arrow] (opt.south)    -- node[lbl, anchor=east]  {No}  (terminate.north);
\end{tikzpicture}
\caption{Overview of adaptive bucket-based pricing.}
\label{fig:flowchart}
\end{figure}
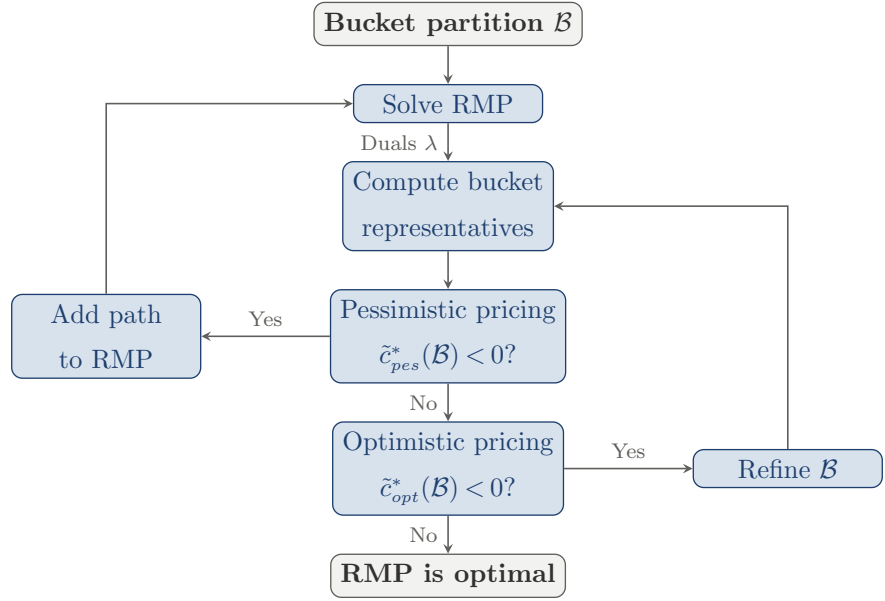

We formalize the notions of bucket and representative in Section~\ref{subsec:bucket}, and introduce the bucket graph in Section~\ref{subsec:bucket_graph}. We describe the pessimistic and optimistic pricing steps in Section~\ref{subsec:bucket_pricing}, and the refinement procedure in Section~\ref{subsec:bucket_refinement}. We prove convergence results in Section~\ref{subsec:bucket_convergence}. We discuss design choices and acceleration techniques in Section~\ref{subsec:bucket_design}, and compatibility with common branch-price-and-cut approaches in Section~\ref{subsec:bucket_extensions}.

\subsection{Buckets and Representatives}
\label{subsec:bucket}

We use interval-based buckets to partition the set of feasible subpaths per block along their path-resource consumption profiles. In particular, we partition $\mbb{Z}^{R_0}$ into a finite set of buckets $\mathcal{B}_i$ for each block $K_i$, where each bucket $B \in \mathcal{B}_i$ is a resource interval $B = [\mathbf{l}_{B}, \mathbf{u}_{B}] \subseteq \mbb{Z}^{R_0}$. Note that these intervals are designed to define bucket restrictions within the algorithm, as opposed to the admissible sets defined in Section~\ref{sec:problem} to define $R$-feasibility in the problem. The interval implicitly defines a set of allowed subpaths in block $i$ whose path-resource consumption vector lies in $B$, i.e.:
\begin{align}
    \mathcal{S}_{B} = \{ \sigma \in \mathcal{S}_i : \mathbf{t}^{R_0}_{\sigma} \in B \}
\end{align}
Among all allowed subpaths, we select the bucket representative $\sigma(B)$ with minimum reduced cost:
\begin{align}
    \sigma(B) = \arg \min_{\sigma \in \mathcal{S}_{B}} \tilde{c}_\sigma,
\end{align}
where the reduced cost of a subpath is defined analogously to that of a regular path. We assume that efficient, problem-specific algorithms for determining bucket representatives are available. 

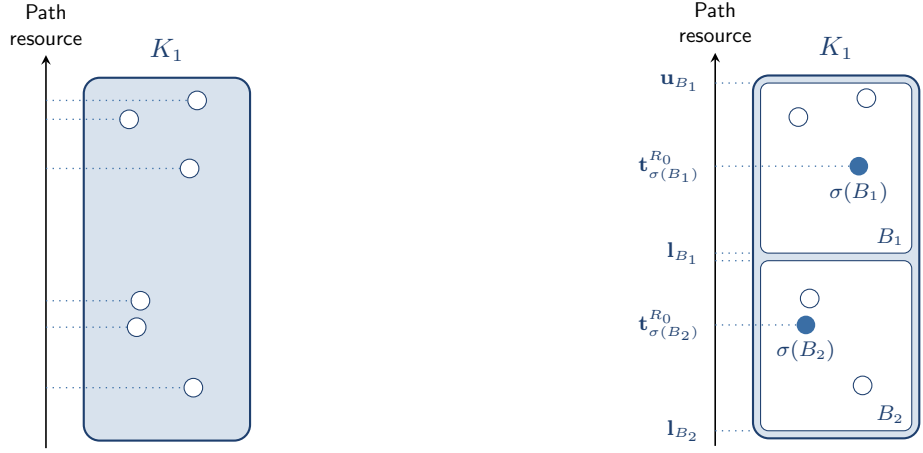
\begin{figure}[h!]
\centering
\begin{subfigure}[c]{0.48\textwidth}
\centering
\begin{tikzpicture}
  \draw[->, >=stealth, black, line width=0.6pt]
    (-1.6,-2.50) -- (-1.6, 2.70)
    node[above, font=\scriptsize\sffamily, align=center] {Path\\resource};
  \Block{0}{0}{K1}{K_1}
  \Subpath{pa}{(-0.50, 1.85)}
  \Subpath{pb}{( 0.30, 1.20)}
  \Subpath{pc}{( 0.40, 2.10)}
  \Subpath{pd}{(-0.35,-0.55)}
  \Subpath{pe}{( 0.35,-1.70)}
  \Subpath{pf}{(-0.40,-0.90)}
  \foreach \x/\y in {-0.50/1.85, 0.30/1.20, 0.40/2.10,
                     -0.35/-0.55, 0.35/-1.70, -0.40/-0.90} {
    \draw[medBlue, dotted, line width=0.5pt, shorten >=3pt]
      (-1.6, \y) -- (\x, \y);
  }
\end{tikzpicture}
\caption{Subpaths along path resource axis.}
\end{subfigure}%
\hfill
\begin{subfigure}[c]{0.48\textwidth}
\centering
\begin{tikzpicture}
  \draw[->, >=stealth, black, line width=0.6pt]
    (-1.6,-2.50) -- (-1.6, 2.70)
    node[above, font=\scriptsize\sffamily, align=center] {Path\\resource};
  \Block{0}{0}{K1}{K_1}
  \node[bkt, minimum width=2.0cm, minimum height=2.25cm] (b1box) at (0, 1.175) {};
  \node[font=\scriptsize\sffamily, darkBlue, anchor=south east, inner sep=3pt]
    at (b1box.south east) {$B_1$};
  \Subpath{qa}{(-0.50, 1.85)}
  \Subpath{qc}{( 0.40, 2.10)}
  \Rep{qB1}{( 0.30, 1.20)}{}
  \node[bkt, minimum width=2.0cm, minimum height=2.25cm] (b2box) at (0,-1.175) {};
  \node[font=\scriptsize\sffamily, darkBlue, anchor=south east, inner sep=3pt]
    at (b2box.south east) {$B_2$};
  \Subpath{qd}{(-0.35,-0.55)}
  \Subpath{qf}{( 0.35,-1.70)}
  \Rep{qB2}{(-0.40,-0.90)}{}
  \draw[medBlue, dotted, line width=0.5pt, shorten >=3pt]
    (-1.6, 1.20) -- (0.30, 1.20);
  \node[font=\scriptsize\sffamily, darkBlue, left=2pt] at (-1.6, 1.20) {$\mb{t}^{R_0}_{\sigma(B_1)}$};
  \draw[medBlue, dotted, line width=0.5pt, shorten >=3pt]
    (-1.6,-0.90) -- (-0.4,-0.90);
  \node[font=\scriptsize\sffamily, darkBlue, left=2pt] at (-1.6,-0.90) {$\mb{t}^{R_0}_{\sigma(B_2)}$};
  \draw[medBlue, dotted, line width=0.5pt] (-1.6, 2.30) -- (-1.1, 2.30);
  \node[font=\scriptsize\sffamily, darkBlue, left=2pt] at (-1.6, 2.30) {$\mb{u}_{B_1}$};
  \draw[medBlue, dotted, line width=0.5pt] (-1.6, 0.05) -- (-1.1, 0.05);
  \node[font=\scriptsize\sffamily, darkBlue, left=2pt] at (-1.6, 0.05) {$\mb{l}_{B_1}$};
  \draw[medBlue, dotted, line width=0.5pt] (-1.6,-0.05) -- (-1.1,-0.05);
  \draw[medBlue, dotted, line width=0.5pt] (-1.6,-2.30) -- (-1.1,-2.30);
  \node[font=\scriptsize\sffamily, darkBlue, left=2pt] at (-1.6,-2.30) {$\mb{l}_{B_2}$};
  \node[font=\scriptsize\sffamily, darkBlue, below=2pt] at (qB1) {$\sigma(B_1)$};
  \node[font=\scriptsize\sffamily, darkBlue, below=2pt] at (qB2) {$\sigma(B_2)$};
\end{tikzpicture}
\caption{Subpaths partitioned into two buckets.}
\end{subfigure}
\caption{Bucket partitioning of subpaths in block $K_1$.}
\label{fig:bucket}
\end{figure}

Figure~\ref{fig:bucket} illustrates how the subpaths of block $K_1$ are partitioned into buckets $B_1$ and $B_2$ based on their path-resource consumption. From the definitions above, $[\mathbf{l}_B, \mathbf{u}_B]$ denotes the resource interval of bucket $B$, the representative $\sigma(B)$ is the minimum reduced cost subpath in $B$, and {$\mathbf{t}^{R_0}_{\sigma(B)}$} denotes the corresponding path-resource consumption vector.

We illustrate the notion of buckets and representatives on the two motivating examples. In the B-MPCVRP (Example~\ref{ex:cvrp}), the path resource tracks cumulative route distance. We partition the routes of each day into one-dimensional buckets, grouping routes whose distance falls within the same interval. The representative is the minimum reduced cost route whose distance is contained in the interval. Finding the representative involves an $\mathcal{NP}$-hard resource-constrained shortest path problem, which can be solved using bidirectional labeling \citep{righini2006symmetry}.

In the RCSP (Example~\ref{ex:rcsp}), the path resource tracks the template length, defined as the difference between end and start time of each duty. We partition duties into two-dimensional buckets specifying allowed start and end time intervals. The representative $\sigma(B)$ is the minimum reduced cost duty respecting these intervals. Finding a representative duty is similar to the pricing problem for a standard crew scheduling problem and can be solved in polynomial time \Citep{rossum2022fast}.

\subsection{Bucket Graph}
\label{subsec:bucket_graph}

The bucket graph is a coarsened representation of the subpath graph $H$, obtained by replacing each bucket's subpaths with a single representative. Specifically, given a bucket partitioning $\mathcal{B} = (\mathcal{B}_1, \dots, \mathcal{B}_n)$, let $\mathcal{S}_{\mathcal{B}} = \bigcup_{B \in \mathcal{B}} \{ \sigma(B) \}$ denote the set of all bucket representatives. We define the bucket graph $H_{\mathcal{B}} = (\{s, t\} \cup \mathcal{S}_{\mathcal{B}}, E_{\mathcal{B}})$ as the graph on these representatives, where an arc $(\sigma(B), \sigma(B')) \in E_{\mathcal{B}}$ exists whenever $B \in \mathcal{B}_i$ and $B' \in \mathcal{B}_{i+1}$ for some $i = 1, \dots, n-1$. By full connectivity of consecutive blocks, all such pairs of representatives can be concatenated. Figure~\ref{fig:bucket_graph} illustrates the bucket graph corresponding to partitioning $\mathcal{B} = (\mathcal{B}_1, \mathcal{B}_2, \mathcal{B}_3)$ with $\mathcal{B}_1 = (B_1, B_2)$, $\mathcal{B}_2 = (B_3, B_4)$, and $\mathcal{B}_3 = (B_5, B_6)$. By design, the bucket graph is much sparser than the original subpath graph of Figure~\ref{fig:subpath_graph}. Yet, the more granular the bucket partition, the closer $H_{\mathcal{B}}$ gets to the full subpath graph $H$. 

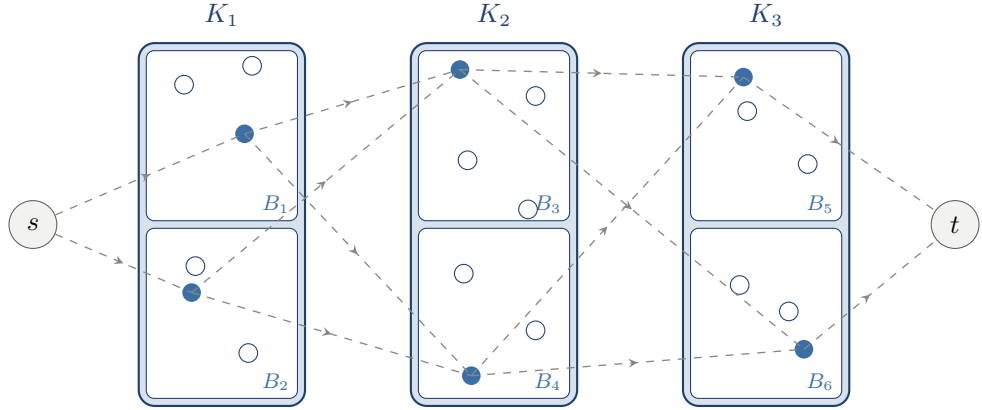
\begin{figure}[h!]
\centering
\begin{tikzpicture}
  \SourceSink{-2.5}{9.7}
  \Block{0}{0}{K1}{K_1}
  \node[bkt, minimum width=2.0cm, minimum height=2.25cm] (a1box) at (0, 1.175) {};
  \node[font=\scriptsize\sffamily, medBlue, anchor=south east, inner sep=3pt]
    at (a1box.south east) {$B_1$};
  \Subpath{pa1}{(-0.50, 1.85)}
  \Rep{qa1}{( 0.30, 1.20)}{}
  \Subpath{pa3}{( 0.40, 2.10)}
  \node[bkt, minimum width=2.0cm, minimum height=2.25cm] (a2box) at (0,-1.175) {};
  \node[font=\scriptsize\sffamily, medBlue, anchor=south east, inner sep=3pt]
    at (a2box.south east) {$B_2$};
  \Subpath{pa4}{(-0.35,-0.55)}
  \Subpath{pa5}{( 0.35,-1.70)}
  \Rep{qa2}{(-0.40,-0.90)}{}
  \Block{3.6}{0}{K2}{K_2}
  \node[bkt, minimum width=2.0cm, minimum height=2.25cm] (b1box) at (3.6, 1.175) {};
  \node[font=\scriptsize\sffamily, medBlue, anchor=south east, inner sep=3pt]
    at (b1box.south east) {$B_3$};
  \Rep{qb1}{(3.15, 2.05)}{}
  \Subpath{pb2}{(4.15, 1.70)}
  \Subpath{pb3}{(3.25, 0.85)}
  \Subpath{pb4}{(4.05, 0.20)}
  \node[bkt, minimum width=2.0cm, minimum height=2.25cm] (b2box) at (3.6,-1.175) {};
  \node[font=\scriptsize\sffamily, medBlue, anchor=south east, inner sep=3pt]
    at (b2box.south east) {$B_4$};
  \Subpath{pb5}{(3.20,-0.65)}
  \Subpath{pb6}{(4.15,-1.40)}
  \Rep{qb2}{(3.30,-2.00)}{}
  \Block{7.2}{0}{K3}{K_3}
  \node[bkt, minimum width=2.0cm, minimum height=2.25cm] (c1box) at (7.2, 1.175) {};
  \node[font=\scriptsize\sffamily, medBlue, anchor=south east, inner sep=3pt]
    at (c1box.south east) {$B_5$};
  \Rep{qc1}{(6.90, 1.95)}{}
  \Subpath{pc2}{(7.75, 0.80)}
  \Subpath{pc3}{(6.95, 1.50)}
  \node[bkt, minimum width=2.0cm, minimum height=2.25cm] (c2box) at (7.2,-1.175) {};
  \node[font=\scriptsize\sffamily, medBlue, anchor=south east, inner sep=3pt]
    at (c2box.south east) {$B_6$};
  \Subpath{pc4}{(6.85,-0.80)}
  \Subpath{pc5}{(7.50,-1.15)}
  \Rep{qc2}{(7.70,-1.65)}{}
  \foreach \i in {qa1,qa2}        { \draw[grarc] (s)  -- (\i); }
  \foreach \i in {qa1,qa2} {
    \foreach \j in {qb1,qb2}      { \draw[grarc] (\i) -- (\j); } }
  \foreach \i in {qb1,qb2} {
    \foreach \j in {qc1,qc2}      { \draw[grarc] (\i) -- (\j); } }
  \foreach \i in {qc1,qc2}        { \draw[grarc] (\i) -- (t); }
\end{tikzpicture}
\caption{Bucket graph for blocks $K_1$, $K_2$, and $K_3$ with two buckets each.
Open circles are subpaths, filled circles are bucket representatives. Inter-block
arcs connect only the bucket representatives.}
\label{fig:bucket_graph}
\end{figure}

\subsection{Pessimistic and Optimistic Pricing}
\label{subsec:bucket_pricing}

Both the pessimistic and optimistic pricing steps apply a labeling algorithm on the bucket graph $H_{\mathcal{B}}$, but differ in how the path resource consumption of each bucket representative is treated. In our implementation, we use bidirectional labeling \citep{righini2006symmetry} for both pricing steps.

In pessimistic pricing, we search for a minimum reduced cost path in $H_{\mathcal{B}}$ using the true resource consumption of each representative. Formally, we solve
\begin{align}
    \tilde{c}^{*}_{\text{pes}}(\mathcal{B}) = \min_{\pi \in \mathcal{P}(H_{\mathcal{B}}, R_0)} \tilde{c}_\pi,
\end{align}
where we recall that $\mathcal{P}(H_{\mathcal{B}}, R_0)$ is the set of $R_0$-feasible paths in $H_{\mathcal{B}}$. Since $H_{\mathcal{B}}$ is restricted to a single representative subpath per bucket, its optimal value forms an upper bound on the true minimum reduced cost over all paths in $\mathcal{P}$. If the optimal value is negative, the corresponding path is added to the RMP. If not, we proceed to optimistic pricing.

In optimistic pricing, we aim to determine whether the failure of pessimistic pricing is due to (i) the coarseness of $H_{\mathcal{B}}$, which excludes subpaths, or (ii) the fact that the current RMP solution is already optimal for the LP relaxation of model~\eqref{eq:master_ip}. To this end, we replace the path-resource consumption vector $\mathbf{t}^{R_0}_{\sigma(B)}$ of each representative $\sigma(B)$ by the lower bound $\mathbf{l}_{B}$, and solve
\begin{align}
    \tilde{c}^{*}_{\text{opt}}(\mathcal{B}) = \min_{\pi \in \mathcal{P}(H_{\mathcal{B}}, R^{\mathcal{B}}_0)} \tilde{c}_\pi,
\end{align}
where $R^{\mathcal{B}}_0$ denotes the resource constraints obtained from $R_0$ by replacing, for each bucket $B \in \mathcal{B}$, the contribution $\mathbf{t}^{r}_{\sigma(B)}$ of $\sigma(B)$ to every path resource $r \in R_0$ by the corresponding bucket lower bound $\mathbf{l}^{r}_{B}$. The value of resource $r$ at the end of any block $i$ along a path $\pi = (s, \sigma(B_1), \dots, \sigma(B_n), t)$ in $\mathcal{P}(H_{\mathcal{B}}, R^{\mathcal{B}}_0)$ then becomes $F^{r}_{i}(\mathbf{l}^{r}_{B_1}, \dots, \mathbf{l}^{r}_{B_i})$ rather than $F^{r}_{i}(\mathbf{t}^{r}_{\sigma(B_1)}, \dots, \mathbf{t}^{r}_{\sigma(B_i)})$. In the B-MPCVRP (Example~\ref{ex:cvrp}), optimistic pricing underestimates the distance of every daily route as its bucket's lower bound. In the RCSP (Example~\ref{ex:rcsp}), optimistic pricing computes a template length considering the latest admissible start time and the earliest admissible end time within each bucket (which amounts to underestimating each resource dimension: end time and minus start time).

By Assumption~\ref{assump:path_resources}, underestimating the resource consumption of each representative can only expand the feasible set: $F^{r}_{i}(\mathbf{l}^{r}_{B_1}, \dots, \mathbf{l}^{r}_{B_i}) \preceq F^{r}_{i}(\mathbf{t}^{r}_{\sigma(B_1)}, \dots, \mathbf{t}^{r}_{\sigma(B_i)})$, and feasibility of the latter vector implies feasibility of the former. Therefore, optimistic pricing returns a lower bound on the true minimum reduced cost over all paths in $\mathcal{P}$. If this lower bound is nonnegative, no negative reduced cost column exists and the column generation algorithm terminates. Otherwise, the optimistic path identifies a sequence of buckets along which a feasible negative reduced cost path \emph{may} exist, and we proceed to the refinement step. In this sense, the optimistic path directs the refinement toward the part of the bucket graph where a negative reduced cost column is most likely to be found.

From the above discussion, it is clear that the partitioning scheme maintains valid lower and upper bounds on the minimum reduced cost.
\begin{prop}[Reduced cost bounds]
\label{prop:rc_bounds}
For any bucket partitioning $\mathcal{B}$, it holds that
\begin{align}
    \tilde{c}_{\emph{opt}}(\mathcal{B}) \leq \min_{\pi \in \mathcal{P}(G, R)} \tilde{c}_{\pi} \leq \tilde{c}_{\emph{pes}}(\mathcal{B}).
\end{align}
\end{prop}
\proof{Proof.}
First consider the upper bound. Every path in $\mathcal{P}(H_{\mathcal{B}}, R_0)$ is a feasible path in $\mathcal{P}(G, R)$, since each bucket representative is a feasible subpath by construction and any two representatives from consecutive blocks can be concatenated. Hence $\tilde{c}_{\text{pes}}(\mathcal{B})$ is the reduced cost of a feasible path, and therefore an upper bound on the minimum reduced cost over $\mathcal{P}(G, R)$.

Now consider the lower bound. Let $\pi^* \in \mathcal{P}(G, R)$ be a path of minimum reduced cost, and write $\pi^* = (s, \sigma^*_1, \dots, \sigma^*_n, t)$ as its unique decomposition into subpaths. For each $i = 1, \dots, n$, let $B_i \in \mathcal{B}_i$ be the unique bucket such that $\sigma^*_i \in \mathcal{S}_{B_i}$. By definition of its representative, $\tilde{c}_{\sigma(B_i)} \leq \tilde{c}_{\sigma^*_i}$, so the path $\hat{\pi} = (s, \sigma(B_1), \dots, \sigma(B_n), t)$ satisfies $\tilde{c}_{\hat{\pi}} \leq \tilde{c}_{\pi^*}$. It remains to show that $\hat{\pi}$ is feasible with respect to the relaxed resource constraints $R^{\mathcal{B}}_0$. Since $\sigma^*_i \in \mathcal{S}_{B_i}$, the contribution vector $\mathbf{t}^{r}_{\sigma^*_i}$ of $\sigma^*_i$ to each path resource $r \in R_0$ satisfies $\mathbf{l}^{r}_{B_i} \preceq \mathbf{t}^{r}_{\sigma^*_i}$. By monotone separable aggregation (Assumption~\ref{assump:path_resources}(ii)), it follows that $F^{r}_{i}(\mathbf{l}^{r}_{B_1}, \dots, \mathbf{l}^{r}_{B_i}) \preceq F^{r}_{i}(\mathbf{t}^{r}_{\sigma^*_1}, \dots, \mathbf{t}^{r}_{\sigma^*_i})$ componentwise for every block $i$ and every path resource $r$. Combined with the downward closure of the feasible set (Assumption~\ref{assump:path_resources}(i)), feasibility of $\pi^*$ under $R_0$ implies feasibility of $\hat{\pi}$ under $R^{\mathcal{B}}_0$. Hence $\tilde{c}_{\text{opt}}(\mathcal{B}) \leq \tilde{c}_{\hat{\pi}} \leq \tilde{c}_{\pi^*} = \min_{\pi \in \mathcal{P}(G, R)} \tilde{c}_\pi$. 
\hfill\Halmos

Proposition~\ref{prop:rc_bounds} shows that the pessimistic and optimistic pricing steps yield valid bounds on the true minimum reduced cost at every iteration. In summary, if $\tilde{c}_{\text{opt}}(\mathcal{B}) \geq 0$, no negative reduced cost column exists and column generation terminates; if $\tilde{c}_{\text{pes}}(\mathcal{B}) < 0$, a negative reduced cost column has been found; the remaining case, $\tilde{c}_{\text{opt}}(\mathcal{B}) < 0 \leq \tilde{c}_{\text{pes}}(\mathcal{B})$, triggers refinement.

\subsection{Refinement}
\label{subsec:bucket_refinement}

By construction, the optimistic path relies on underestimated resource consumptions and is not truly feasible with respect to the path resource constraints $R_0$. The goal of refinement is to partition one or more buckets along this path into smaller buckets, thereby eliminating the incumbent optimistic path from $H_{\mathcal{B}}$ and reducing the gap between pessimistic and optimistic pricing.

We formalize the notion of refinement as follows.
\begin{defin}[Refinement]
\label{def:refinement}
Let $\mathcal{B} = (\mathcal{B}_1, \dots, \mathcal{B}_n)$ denote a bucket partitioning. We say $\mathcal{B}' = (\mathcal{B}'_1, \dots, \mathcal{B}'_n)$ is a refinement of $\mathcal{B}$ when, for all $i = 1, \dots, n$, it holds that
\begin{enumerate}[(i)]
    \item $|\mathcal{B}'_i| \geq |\mathcal{B}_i|$, with at least one inequality strict,
    \item $\bigcup\limits_{B' \in \mathcal{B}'_i} B' = \bigcup\limits_{B \in \mathcal{B}_i} B$, and
    \item for every $B' \in \mathcal{B}'_i$ there exists $B \in \mathcal{B}_i$ such 
    that $B' \subseteq B$.
\end{enumerate}
\end{defin}
These conditions state that a refinement strictly increases the number of buckets in at least one block, preserves the union of bucket intervals within each block, and only splits existing buckets without merging them. Note that the refinement procedure is not unique, which leaves flexibility to choose which buckets to split and along which resource dimension(s). Figure~\ref{fig:refinement} shows a possible refinement of the bucket graph in Figure~\ref{fig:bucket_graph}.

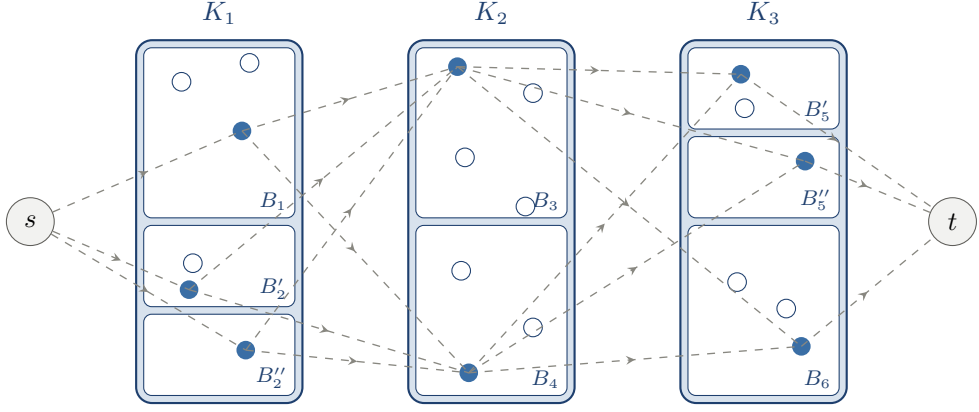
\begin{figure}[htbp!]
\centering
\tikzset{
  arclbl/.style={font=\scriptsize, bgray800,
                 fill=white, draw=bgray400, line width=0.4pt,
                 inner sep=2pt, rounded corners=2pt},
}
\begin{tikzpicture}

  \useasboundingbox (-3.2, -3.0) rectangle (11.5, 3.0);
  \SourceSink{-2.5}{9.7}

  \Block{0}{0}{K1}{K_1}
  \node[bkt, minimum width=2.0cm, minimum height=2.25cm] (a1box) at (0, 1.175) {};
  \node[font=\scriptsize\sffamily, darkBlue, anchor=south east, inner sep=3pt]
    at (a1box.south east) {$B_1$};
  \Subpath{pa1a}{(-0.50, 1.85)}
  \Rep{qa1}{( 0.30, 1.20)}{}
  \Subpath{pa1b}{( 0.40, 2.10)}
  \node[bkt, minimum width=2.0cm, minimum height=1.075cm] (a2pbox) at (0,-0.5875) {};
  \node[font=\scriptsize\sffamily, darkBlue, anchor=south east, inner sep=3pt]
    at (a2pbox.south east) {$B_2'$};
  \Subpath{pa2pa}{(-0.35,-0.55)}
  \Rep{qa2p}{(-0.40,-0.90)}{}
  
  \node[bkt, minimum width=2.0cm, minimum height=1.075cm] (a2ppbox) at (0,-1.7625) {};
  \node[font=\scriptsize\sffamily, darkBlue, anchor=south east, inner sep=3pt]
    at (a2ppbox.south east) {$B_2''$};
  \Rep{qa2pp}{( 0.35,-1.70)}{}

  \Block{3.6}{0}{K2}{K_2}
  \node[bkt, minimum width=2.0cm, minimum height=2.25cm] (b1box) at (3.6, 1.175) {};
  \node[font=\scriptsize\sffamily, darkBlue, anchor=south east, inner sep=3pt]
    at (b1box.south east) {$B_3$};
  \Rep{qb1}{(3.15, 2.05)}{}
  \Subpath{pb1a}{(4.15, 1.70)}
  \Subpath{pb1b}{(3.25, 0.85)}
  \Subpath{pb1c}{(4.05, 0.20)}
  \node[bkt, minimum width=2.0cm, minimum height=2.25cm] (b2box) at (3.6,-1.175) {};
  \node[font=\scriptsize\sffamily, darkBlue, anchor=south east, inner sep=3pt]
    at (b2box.south east) {$B_4$};
  \Subpath{pb2a}{(3.20,-0.65)}
  \Subpath{pb2b}{(4.15,-1.40)}
  \Rep{qb2}{(3.30,-2.00)}{}

  \Block{7.2}{0}{K3}{K_3}
  \node[bkt, minimum width=2.0cm, minimum height=1.075cm] (c1pbox) at (7.2, 1.7625) {};
  \node[font=\scriptsize\sffamily, darkBlue, anchor=south east, inner sep=3pt]
    at (c1pbox.south east) {$B_5'$};
  \Rep{qc1p}{(6.90, 1.95)}{}
  \Subpath{pc1pa}{(6.95, 1.50)}
  \node[bkt, minimum width=2.0cm, minimum height=1.075cm] (c1ppbox) at (7.2, 0.5875) {};
  \node[font=\scriptsize\sffamily, darkBlue, anchor=south east, inner sep=3pt]
    at (c1ppbox.south east) {$B_5''$};
  \Rep{qc1pp}{(7.75, 0.80)}{}
  \node[bkt, minimum width=2.0cm, minimum height=2.25cm] (c2box) at (7.2,-1.175) {};
  \node[font=\scriptsize\sffamily, darkBlue, anchor=south east, inner sep=3pt]
    at (c2box.south east) {$B_6$};
  \Subpath{pc2a}{(6.85,-0.80)}
  \Subpath{pc2b}{(7.50,-1.15)}
  \Rep{qc2}{(7.70,-1.65)}{}

  \foreach \i in {qa1,qa2p,qa2pp}      { \draw[grarc] (s)  -- (\i); }
  \foreach \i in {qa1,qa2p,qa2pp} {
    \foreach \j in {qb1,qb2}           { \draw[grarc] (\i) -- (\j); } }
  \foreach \i in {qb1,qb2} {
    \foreach \j in {qc1p,qc1pp,qc2}    { \draw[grarc] (\i) -- (\j); } }
  \foreach \i in {qc1p,qc1pp,qc2}      { \draw[grarc] (\i) -- (t); }

\end{tikzpicture}
\caption{Refined bucket graph after splitting $B_2$ into $B_2'$ and $B_2''$ in $K_1$,and $B_5$ into $B_5'$ and $B_5''$ in $K_3$.}
\label{fig:refinement}
\end{figure}

The reduced cost bounds of Proposition~\ref{prop:rc_bounds} tighten monotonically as a result of refinement.
\begin{prop}[Monotonicity under refinement]
\label{prop:refinement}
Let $\mathcal{B}'$ be a refinement of $\mathcal{B}$. Then
\begin{align}
    \tilde{c}_{\emph{pes}}(\mathcal{B}') \leq \tilde{c}_{\emph{pes}}(\mathcal{B}) \quad \text{ and } \quad \tilde{c}_{\emph{opt}}(\mathcal{B}') \geq \tilde{c}_{\emph{opt}}(\mathcal{B}).
\end{align}
\end{prop}

\proof{Proof.}
We start by proving the inequality on pessimistic pricing. Consider any path $\pi \in \mathcal{P}(H_{\mathcal{B}}, R_0)$, written as $(s, \sigma(B_1), \dots, \sigma(B_n), t)$. For each $i = 1, \dots, n$, the representative $\sigma(B_i)$ has a path resource consumption vector in $B_i$, and by conditions~(ii) and~(iii) of Definition~\ref{def:refinement} there is a unique bucket $B'_i \in \mathcal{B}'_i$ with $B'_i \subseteq B_i$ whose interval contains this vector. Since $\sigma(B_i)$ has minimum reduced cost over $\mathcal{S}_{B_i}$, it is optimal over the smaller set $\mathcal{S}_{B_{i}^{'}}$, i.e., $\sigma(B_{i}^{'}) = \sigma(B_i)$ and $\tilde{c}_{\sigma(B'_i)} = \tilde{c}_{\sigma(B_i)}$. The path $(\sigma(B'_1), \dots, \sigma(B'_n))$ therefore belongs to $\mathcal{P}(H_{\mathcal{B}'}, R_0)$ with reduced cost equal to $\tilde{c}_\pi$. This proves that $\tilde{c}_{\text{pes}}(\mathcal{B}') \leq \tilde{c}_{\text{pes}}(\mathcal{B})$.

Now consider the inequality on optimistic pricing. Consider any path $\pi' \in \mathcal{P}(H_{\mathcal{B}'}, R^{\mathcal{B}'}_0)$, written as $(s, \sigma(B'_1), \dots, \sigma(B'_n), t)$. By condition~(iii) of Definition~\ref{def:refinement}, each $B'_i$ is contained in some $B_i \in \mathcal{B}_i$, so $\mathbf{l}^{r}_{B_i} \preceq \mathbf{l}^{r}_{B'_i}$ for every path resource $r \in R_0$. Hence, the resource lower bound vector used in the optimistic relaxation of $\mathcal{B}$ is componentwise no larger than that of $\mathcal{B}'$. By monotone separable aggregation (Assumption~\ref{assump:path_resources}(ii)), $F^{r}_{i}(\mathbf{l}^{r}_{B_1}, \dots, \mathbf{l}^{r}_{B_i}) \preceq F^{r}_{i}(\mathbf{l}^{r}_{B'_1}, \dots, \mathbf{l}^{r}_{B'_i})$ componentwise for every block $i$ and every path resource $r$. Combined with downward closure (Assumption~\ref{assump:path_resources}(i)), feasibility of $\pi'$ under $R^{\mathcal{B}'}_0$ implies feasibility of the path $(s, \sigma(B_1), \dots, \sigma(B_n), t)$ under $R^{\mathcal{B}}_0$. Moreover, $\tilde{c}_{\sigma(B_i)} \leq \tilde{c}_{\sigma(B'_i)}$ by optimality of $\sigma(B_i)$ over the larger set $\mathcal{S}_{B_i} \supseteq \mathcal{S}_{B'_i}$, so the reduced cost of the constructed path does not exceed $\tilde{c}_{\pi'}$. Hence $\tilde{c}_{\text{opt}}(\mathcal{B}) \leq \tilde{c}_{\pi'}$, so that $\tilde{c}_{\text{opt}}(\mathcal{B}) \leq \tilde{c}_{\text{opt}}(\mathcal{B}')$.
\hfill\Halmos

\subsection{Convergence}
\label{subsec:bucket_convergence}

We now establish our main result that the adaptive partitioning algorithm solves the pricing problem to optimality in a finite number of refinement steps.

\begin{prop}[Convergence]
\label{prop:local_convergence}
Within a finite number of refinement steps, the adaptive partitioning algorithm either returns a negative reduced cost column or certifies that none exists.
\end{prop}

\proof{Proof.}
If $\tilde{c}_{\text{pes}}(\mathcal{B}) < 0$, the algorithm immediately returns a negative reduced cost column. If $\tilde{c}_{\text{opt}}(\mathcal{B}) \geq 0$, it follows from Proposition~\ref{prop:rc_bounds} that no negative reduced cost column exists, and the algorithm terminates. Otherwise, the algorithm refines $\mathcal{B}$. We argue that this case can occur only finitely many times.

Each refinement step splits at least one bucket, strictly increasing $|\mathcal{B}|$. By Definition~\ref{def:refinement}, a bucket $B$ can only be split when it contains at least one feasible subpath, i.e., $\mathcal{S}_B \neq \emptyset$. Since the set $\mathcal{S}$ of feasible subpaths is finite by assumption, the total number of non-empty buckets is bounded by $|\mathcal{S}|$. After at most $|\mathcal{S}|$ refinement steps, every bucket contains subpaths sharing a single path-resource consumption vector, so that each bucket representative is the unique minimum reduced cost subpath in its bucket. At this point, the bucket graph $H_{\mathcal{B}}$ contains all non-dominated subpaths, the pessimistic and optimistic bounds coincide, and $\tilde{c}_{\text{pes}}(\mathcal{B}) = \min_{\pi \in \mathcal{P}(G, R)} \tilde{c}_\pi$. The algorithm then either returns a negative reduced cost column or correctly certifies optimality.
\hfill\Halmos

Proposition~\ref{prop:local_convergence} does not exclude the possibility that the algorithm reverts to a full enumeration of non-dominated subpaths in the worst case. In our computational experiments, however, the adaptive partitioning algorithm typically outperforms the enumerative benchmark.

Finally, finite convergence of the pricing algorithm translates to finite convergence of the overall column generation algorithm.

\begin{prop}[Column generation exactness]
\label{prop:global_convergence}
Column generation with the adaptive partitioning algorithm converges to the optimal solution of the LP relaxation of model~\eqref{eq:master_ip} in a finite number of column generation iterations.
\end{prop}

\proof{Proof.} 
By Proposition~\ref{prop:local_convergence}, each call to the pricing algorithm either adds a negative reduced cost column to the RMP or correctly concludes that no such column exists, in which case the current RMP solution is optimal for the LP relaxation of model~\eqref{eq:master_ip}. In the former case, a new column is added to the RMP. Since the set $\mathcal{P}$ of feasible paths is finite, the number of columns that can be added is finite, and the algorithm finds the optimal solution in a finite number of iterations.
\hfill\Halmos

\subsection{Design Choices and Acceleration Techniques}
\label{subsec:bucket_design}

The adaptive partitioning algorithm involves several design choices. We discuss the most important ones, as well as some acceleration techniques, and evaluate them in Sections~\ref{sec:cvrp}-\ref{sec:crew}. 

\paragraph{Initial bucket granularity.}

A coarser initialization reduces the upfront cost of computing bucket representatives, but may require more refinement iterations before convergence to a path of minimum reduced cost. A finer initialization yields a more accurate approximation of the subpath graph and will lead to more (diverse) columns per iteration. Of course, this comes at a computational cost when applying the labeling algorithms over the finer bucket graph.

\paragraph{Refinement strategy.}

We consider two strategies for splitting buckets during the refinement step, illustrated in Figure~\ref{fig:refinement_strategies} with two-dimensional resources. The \emph{midpoint strategy} splits a bucket exactly in half along every resource axis, regardless of the resource consumption of its representative. The \emph{representative strategy} instead splits a bucket at the resource consumption vector of its current representative, so that the resource consumption of the representative becomes the lower bound of one of the resulting sub-buckets. The representative strategy has the advantage of immediately separating the incumbent representative from the lower sub-bucket, ensuring that the optimistic path that triggered refinement is directly eliminated. The midpoint strategy, by contrast, may require multiple refinement steps to eliminate the same path, but produces more balanced buckets.

%
%

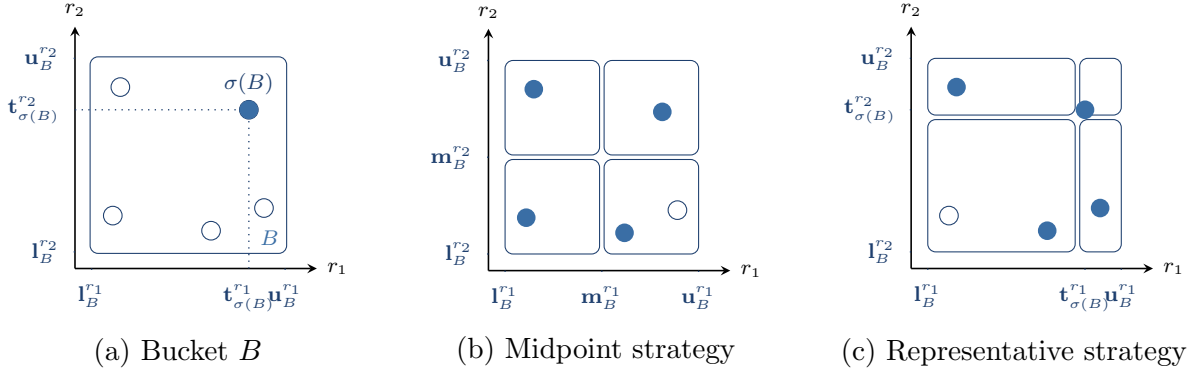
\begin{figure}[htbp]
\centering

\begin{subfigure}[c]{0.33\textwidth}
\centering
\begin{tikzpicture}

  \draw[->, >=stealth, black, line width=0.6pt]
    (0, 0) -- (3.20, 0)
    node[right, font=\scriptsize\sffamily] {$r_1$};
  \draw[->, >=stealth, black, line width=0.6pt]
    (0, 0) -- (0, 3.20)
    node[above, font=\scriptsize\sffamily] {$r_2$};

  \node[bkt, minimum width=2.60cm, minimum height=2.60cm]
    (bbox) at (1.50, 1.50) {};
  \node[font=\scriptsize\sffamily, medBlue, anchor=south east, inner sep=3pt]
    at (bbox.south east) {$B$};

  \draw[darkBlue, dotted, line width=0.5pt] (0.22, 0) -- (0.22, -0.05);
  \node[font=\scriptsize\sffamily, darkBlue, below=2pt] at (0.22, 0) {$\mb{l}^{r_1}_{B}$};
  \draw[darkBlue, dotted, line width=0.5pt] (2.78, 0) -- (2.78, -0.05);
  \node[font=\scriptsize\sffamily, darkBlue, below=2pt] at (2.78, 0) {$\mb{u}^{r_1}_{B}$};
  \draw[darkBlue, dotted, line width=0.5pt] (0, 0.22) -- (-0.05, 0.22);
  \node[font=\scriptsize\sffamily, darkBlue, left=2pt] at (0, 0.22) {$\mb{l}^{r_2}_{B}$};
  \draw[darkBlue, dotted, line width=0.5pt] (0, 2.78) -- (-0.05, 2.78);
  \node[font=\scriptsize\sffamily, darkBlue, left=2pt] at (0, 2.78) {$\mb{u}^{r_2}_{B}$};

  \Subpath{pa}{(0.60, 2.40)}
  \Subpath{pb}{(2.30, 2.10)}
  \Subpath{pc}{(2.50, 0.80)}
  \Subpath{pd}{(0.50, 0.70)}
  \Subpath{pe}{(1.80, 0.50)}

  \Rep{rep}{(2.30, 2.10)}{}
  \node[font=\scriptsize\sffamily, darkBlue, above=2pt] at (rep) {$\sigma(B)$};

  \draw[darkBlue, dotted, line width=0.5pt, shorten >=3pt]
    (2.30, 0) -- (rep);
  \draw[darkBlue, dotted, line width=0.5pt, shorten >=3pt]
    (0, 2.10) -- (rep);
  \node[font=\scriptsize\sffamily, darkBlue, below=2pt] at (2.30, 0)
    {$\mb{t}^{r_1}_{\sigma(B)}$};
  \node[font=\scriptsize\sffamily, darkBlue, left=2pt] at (0, 2.10)
    {$\mb{t}^{r_2}_{\sigma(B)}$};

\end{tikzpicture}
\caption{Bucket $B$}
\end{subfigure}%
\hfill
\begin{subfigure}[c]{0.33\textwidth}
\centering
\begin{tikzpicture}

  \draw[->, >=stealth, black, line width=0.6pt]
    (0, 0) -- (3.20, 0)
    node[right, font=\scriptsize\sffamily] {$r_1$};
  \draw[->, >=stealth, black, line width=0.6pt]
    (0, 0) -- (0, 3.20)
    node[above, font=\scriptsize\sffamily] {$r_2$};

  \node[bkt, minimum width=1.25cm, minimum height=1.25cm] (b11) at (0.845, 0.845) {};
  \node[bkt, minimum width=1.25cm, minimum height=1.25cm] (b12) at (2.155, 0.845) {};
  \node[bkt, minimum width=1.25cm, minimum height=1.25cm] (b21) at (0.845, 2.155) {};
  \node[bkt, minimum width=1.25cm, minimum height=1.25cm] (b22) at (2.155, 2.155) {};

  \draw[darkBlue, dotted, line width=0.5pt] (1.50, 0) -- (1.50, -0.05);
  \node[font=\scriptsize\sffamily, darkBlue, below=2pt] at (1.50, 0) {$\mb{m}^{r_1}_{B}$};
  \draw[darkBlue, dotted, line width=0.5pt] (0, 1.50) -- (-0.05, 1.50);
  \node[font=\scriptsize\sffamily, darkBlue, left=2pt] at (0, 1.50) {$\mb{m}^{r_2}_{B}$};

  \draw[darkBlue, dotted, line width=0.5pt] (0.22, 0) -- (0.22, -0.05);
  \node[font=\scriptsize\sffamily, darkBlue, below=2pt] at (0.22, 0) {$\mb{l}^{r_1}_{B}$};
  \draw[darkBlue, dotted, line width=0.5pt] (2.78, 0) -- (2.78, -0.05);
  \node[font=\scriptsize\sffamily, darkBlue, below=2pt] at (2.78, 0) {$\mb{u}^{r_1}_{B}$};
  \draw[darkBlue, dotted, line width=0.5pt] (0, 0.22) -- (-0.05, 0.22);
  \node[font=\scriptsize\sffamily, darkBlue, left=2pt] at (0, 0.22) {$\mb{l}^{r_2}_{B}$};
  \draw[darkBlue, dotted, line width=0.5pt] (0, 2.78) -- (-0.05, 2.78);
  \node[font=\scriptsize\sffamily, darkBlue, left=2pt] at (0, 2.78) {$\mb{u}^{r_2}_{B}$};

  \Subpath{pc}{(2.50, 0.80)}

  \Rep{repBL}{(0.50, 0.70)}{}
  \Rep{repBR}{(1.80, 0.50)}{}
  \Rep{repTL}{(0.60, 2.40)}{}
  \Rep{repTR}{(2.30, 2.10)}{}

\end{tikzpicture}
\caption{Midpoint strategy}
\end{subfigure}%
\hfill
\begin{subfigure}[c]{0.33\textwidth}
\centering
\begin{tikzpicture}

  \draw[->, >=stealth, black, line width=0.6pt]
    (0, 0) -- (3.20, 0)
    node[right, font=\scriptsize\sffamily] {$r_1$};
  \draw[->, >=stealth, black, line width=0.6pt]
    (0, 0) -- (0, 3.20)
    node[above, font=\scriptsize\sffamily] {$r_2$};

  \node[bkt, minimum width=1.95cm, minimum height=1.75cm] (bl) at (1.195, 1.095) {};
  \node[bkt, minimum width=0.55cm, minimum height=1.75cm] (br) at (2.505, 1.095) {};
  \node[bkt, minimum width=1.95cm, minimum height=0.75cm] (tl) at (1.195, 2.405) {};
  \node[bkt, minimum width=0.55cm, minimum height=0.75cm] (tr) at (2.505, 2.405) {};

  \draw[darkBlue, dotted, line width=0.5pt] (2.30, 0) -- (2.30, -0.05);
  \node[font=\scriptsize\sffamily, darkBlue, below=2pt] at (2.30, 0)
    {$\mb{t}^{r_1}_{\sigma(B)}$};
  \draw[darkBlue, dotted, line width=0.5pt] (0, 2.10) -- (-0.05, 2.10);
  \node[font=\scriptsize\sffamily, darkBlue, left=2pt] at (0, 2.10)
    {$\mb{t}^{r_2}_{\sigma(B)}$};

  \draw[darkBlue, dotted, line width=0.5pt] (0.22, 0) -- (0.22, -0.05);
  \node[font=\scriptsize\sffamily, darkBlue, below=2pt] at (0.22, 0) {$\mb{l}^{r_1}_{B}$};
  \draw[darkBlue, dotted, line width=0.5pt] (2.78, 0) -- (2.78, -0.05);
  \node[font=\scriptsize\sffamily, darkBlue, below=2pt] at (2.78, 0) {$\mb{u}^{r_1}_{B}$};
  \draw[darkBlue, dotted, line width=0.5pt] (0, 0.22) -- (-0.05, 0.22);
  \node[font=\scriptsize\sffamily, darkBlue, left=2pt] at (0, 0.22) {$\mb{l}^{r_2}_{B}$};
  \draw[darkBlue, dotted, line width=0.5pt] (0, 2.78) -- (-0.05, 2.78);
  \node[font=\scriptsize\sffamily, darkBlue, left=2pt] at (0, 2.78) {$\mb{u}^{r_2}_{B}$};

  \Subpath{pd}{(0.50, 0.70)}

  \Rep{repBL}{(1.80, 0.50)}{}
  \Rep{repBR}{(2.50, 0.80)}{}
  \Rep{repTL}{(0.60, 2.40)}{}
  \Rep{repTR}{(2.30, 2.10)}{}

\end{tikzpicture}
\caption{Representative strategy}
\end{subfigure}

\caption{Refinement strategies applied to bucket $B$ in a two-dimensional
path-resource space, with simultaneous splits along all resources (leading to four sub-buckets per refined bucket). The midpoint strategy splits $B$ into four equal sub-buckets, while the representative strategy instead splits at the resource consumption of the current representative $\sigma(B)$.}
\label{fig:refinement_strategies}
\end{figure}

Of course, one can devise more sophisticated refinement strategies than those. Notably, our strategies involve splitting along all resource axes simultaneously, yielding $2^{\Pi_{r \in R_0} d_r}$ sub-buckets per refined bucket. Yet, any strategy yielding at least two sub-buckets is valid, for instance by splitting along a single resource axis or a subset of resource axes. Moreover, one can exploit information about the optimistic path and its resource violations to refine only a single critical bucket along the path or to apply different splitting rules for different resource dimensions.

\paragraph{Reusing representatives across iterations.}

Computing bucket representatives with respect to updated dual variables can be expensive. Since dual variables typically change only moderately between consecutive column generation iterations, one may first perform pessimistic pricing using the representatives from the previous iteration, and only recompute representatives when no negative reduced cost path is found.

\paragraph{Bucket merging.}

While refinement is necessary to guarantee convergence, it also grows the bucket graph and thereby increases the cost of subsequent pricing steps. We counteract this with a merging operation that coarsens $H_{\mathcal{B}}$ by combining adjacent buckets whenever doing so does not exclude any negative reduced cost path. The intuition is that two adjacent buckets contribute little by remaining separate if the optimistic path through their union has nonnegative reduced cost: in that case, the merged bucket would not trigger refinement in the next iteration, and keeping it split only inflates the bucket graph.

Formally, we say that two buckets $B, B' \in \mathcal{B}_i$ are eligible for merging when (i) they are adjacent, i.e., $\mathbf{l}_{B'} = \mathbf{u}_{B} + \mathbf{e_i}$ for some unit vector $\mathbf{e}_i$, and (ii) the merged bucket $B \cup B'$ would not need to be immediately refined in the next iteration. To make condition~(ii) precise, define
\begin{align}
    \tilde{c}^{*}_{\text{opt}}(\mathcal{B}, B) = \min_{\pi \in \mathcal{P}(H_{\mathcal{B}}, R^{\mathcal{B}}_0):\ \sigma(B) \in \pi} \tilde{c}_\pi
\end{align}
as the minimum reduced cost over all optimistic paths passing through bucket $B$. Condition~(ii) is then equivalent to
\begin{align}
    \tilde{c}^{*}_{\text{opt}}(\mathcal{B}, B) + \min\{0, \tilde{c}_{\sigma(B')} - \tilde{c}_{\sigma(B)}\} \geq 0,
\end{align}
where the left-hand side equals the reduced cost of the optimistic path through the bucket obtained by merging $B$ and $B'$.

A slight modification of the bidirectional labeling algorithm yields $\tilde{c}^{*}_{\text{opt}}(\mathcal{B}, B)$ for all buckets $B \in \mathcal{B}$. This is exact but may be computationally expensive. As an alternative, we drop the path resource constraints and verify condition~(ii) using a shortest path criterion on $H_{\mathcal{B}}$. The shortest path criterion may identify fewer mergeable pairs but is more efficient computationally. In our implementation, we adopt the latter approach whenever the number of buckets exceeds a threshold.

\subsection{Extensions}
\label{subsec:bucket_extensions}

As column generation with adaptive partitioning can solve the LP relaxation of~\eqref{eq:master_ip} to optimality, it can be embedded into branch-price-and-cut algorithms for solving~\eqref{eq:master_ip}. To this end, the adaptive partitioning algorithm needs to be compatible with the core elements of branch-price-and-cut algorithms and standard acceleration techniques.

\emph{Heuristic pricing.} One may avoid expensive representative computations by using heuristic algorithms to find representatives, similarly to the reuse of representatives across iterations discussed above. This can significantly reduce the number of exact representative computations required, particularly in the early iterations of column generation.

\emph{Multiple pricing problems.} When the pricing problem decomposes across multiple independent subproblems, for instance due to heterogeneous vehicles or multiple crew depots, one can maintain a dedicated bucket graph per subproblem and solve them in parallel.

\emph{Cutting planes.} The algorithm is compatible with robust cutting planes, i.e., cuts that can be added to the master problem without modifying the pricing problem. 

\emph{Branching.} The algorithm is compatible with branching decisions that can be enforced either at the subpath level, i.e., when generating bucket representatives, or at the path level, i.e., when pricing in the bucket graph. After branching, representatives that were computed under a parent node may become infeasible or suboptimal with respect to the new branching constraints. When reusing representatives from earlier iterations, explicit feasibility checks must therefore be performed to ensure compatibility with the current branching decisions.

\emph{Reduced cost fixing.} Standard reduced cost fixing can be applied both in the element graph, to prune arcs that cannot lead to a negative reduced cost path, and in the bucket graph, to eliminate buckets whose optimistic reduced cost exceeds the current bound \citep{baldacci2008exact}. 

\section{Computational Results on the B-MPCVRP}
\label{sec:cvrp}

We demonstrate the effectiveness of the adaptive partitioning algorithm on the B-MPCVRP (Example~\ref{ex:cvrp}). We solve this problem using a path-based formulation containing a variable for every feasible route schedule. This formulation is similar to model~\eqref{eq:master_ip}, with set partitioning constraints instead of set covering constraints and an additional constraint enforcing that exactly $K$ schedules are selected. We solve the LP relaxation via column generation with the adaptive partitioning algorithm, using a partition of the set of daily routes into buckets specifying distance intervals.

We describe the set-up of our experiments in Section~\ref{subsec:cvrp_setup}. We analyze the effect of several design choices on the performance of the adaptive partitioning algorithm for solving the root node in Section~\ref{subsec:cvrp_strategies}, and compare the best-performing configuration against the enumerative benchmark in Section~\ref{subsec:cvrp_benchmark}. Finally, we compare branch-price-and-cut for the path-based formulation against its counterpart for a subpath-based formulation in Section~\ref{subsec:cvrp_branchPrice}.

\subsection{Set-Up}
\label{subsec:cvrp_setup}

We generate instances based on the 640-customer CVRP benchmark instance \texttt{X-n641-k35} from \citet{uchoa2017new}. For each number of daily customers $n \in \{15, 20, 25\}$ and number of days $T \in \{2, 3, 4\}$, we construct five instances by selecting disjoint sets of size $(n + 1) \cdot T$ from the benchmark instance. We set the number of vehicles to $K = 5$.

We vary the tightness of the maximum total distance constraint $D$. We compute a lower bound $D_{\min}$ by solving a regular CVRP for each day independently and dividing the total cost by $K$. For the upper bound $D_{\max}$, we follow the sequential heuristic of \citet{nekooghadirli2026workload}, which solves the same daily CVRPs but assigns the resulting routes to vehicles to minimize the maximum workload rather than total cost. We set $D = D_{\min} + \delta \cdot (D_{\max} - D_{\min})$ for $\delta \in \{0.1, 0.3, 0.5, 0.7, 0.9\}$, yielding $225$ instances in total, with smaller $\delta$ corresponding to tighter workload constraints.

The column generation procedure is configured as follows. We initialize the RMP with a set of artificial columns. In each pricing iteration, we generate at most 500 columns and add all of them to the RMP. We apply Wentges dual smoothing with parameter $\alpha = 0.5$ \citep{pessoa2018automation}. Every five iterations, we remove all columns that have not been in the basis during the last ten iterations, while maintaining at least 1,000 non-basic columns. We use custom implementations of bidirectional labeling to determine bucket representatives and to perform pessimistic and optimistic pricing in the bucket graph. We find representatives for different buckets in parallel on up to eight threads.

We implement all algorithms in \texttt{Java}, and solve all LPs using \texttt{CPLEX 22.1.0} on six threads. All experiments are conducted on cluster nodes with 16 GB RAM and an AMD Rome 7H12 processor.

\subsection{Configuration of Adaptive Partitioning}
\label{subsec:cvrp_strategies}

We evaluate the adaptive partitioning algorithm across all combinations of three design choices: whether to reuse representatives from the previous column generation iteration (\emph{Reuse}), whether to apply the midpoint or representative refinement strategy (\emph{Midway}), and whether to apply bucket merging (\emph{Merge}). For all eight combinations, we vary the initial bucket width across $\{100, 250, 500\}$ distance units to assess the sensitivity of the algorithm to the initial granularity of the bucket graph. We assess these 24 configurations by solving the root node LP relaxation on the largest instances with $n = 25$ customers per day, $T = 4$ periods, and $\delta = 10\%$.

\begin{figure}[h!]
  \centering
  \begin{subfigure}[b]{4.60cm}
    \centering
    \begin{tikzpicture}[x=1cm, y=1cm]

      \node[font=\tiny, black, left] at (-0.1,0) {0};
      \node[font=\tiny, black, left] at (-0.1,1.25) {75};
      \node[font=\tiny, black, left] at (-0.1,2.50) {150};
      \node[font=\tiny, black, left] at (-0.1,3.75) {225};
      \node[font=\tiny, black, left] at (-0.1,5.00) {300};

      \draw[->, >=stealth, black, line width=0.6pt] (0,-0.1) -- (0,5.40)
        node[above, font=\tiny, black] {Time (s)};

      \draw[fill=medBlue!20, draw=darkBlue!75, line width=0.5pt] (0.00, 0) rectangle (0.32, 2.11);
      \draw[fill=darkBlue!75, draw=darkBlue!75, line width=0.5pt] (0.42, 0) rectangle (0.74, 2.90);
      \node[font=\tiny, black, below=1pt] at (0.37, 0) {100};

      \draw[fill=medBlue!20, draw=darkBlue!75, line width=0.5pt] (1.09, 0) rectangle (1.41, 2.06);
      \draw[fill=darkBlue!75, draw=darkBlue!75, line width=0.5pt] (1.51, 0) rectangle (1.83, 2.53);
      \node[font=\tiny, black, below=1pt] at (1.46, 0) {250};

      \draw[fill=medBlue!20, draw=darkBlue!75, line width=0.5pt] (2.18, 0) rectangle (2.50, 2.82);
      \draw[fill=darkBlue!75, draw=darkBlue!75, line width=0.5pt] (2.60, 0) rectangle (2.92, 2.92);
      \node[font=\tiny, black, below=1pt] at (2.55, 0) {500};

      \draw[->, >=stealth, black, line width=0.6pt] (-0.1,0) -- (3.12,0);
    \end{tikzpicture}
    \caption{No reuse, midway}
  \end{subfigure}
  \hspace{0.5cm}
  \begin{subfigure}[b]{4.60cm}
    \centering
    \begin{tikzpicture}[x=1cm, y=1cm]

      \node[font=\tiny, black, left] at (-0.1,0) {0};
      \node[font=\tiny, black, left] at (-0.1,1.25) {75};
      \node[font=\tiny, black, left] at (-0.1,2.50) {150};
      \node[font=\tiny, black, left] at (-0.1,3.75) {225};
      \node[font=\tiny, black, left] at (-0.1,5.00) {300};

      \draw[->, >=stealth, black, line width=0.6pt] (0,-0.1) -- (0,5.40)
        node[above, font=\tiny, black] {Time (s)};

      \draw[fill=medBlue!20, draw=darkBlue!75, line width=0.5pt] (0.00, 0) rectangle (0.32, 1.88);
      \draw[fill=darkBlue!75, draw=darkBlue!75, line width=0.5pt] (0.42, 0) rectangle (0.74, 2.65);
      \node[font=\tiny, black, below=1pt] at (0.37, 0) {100};

      \draw[fill=medBlue!20, draw=darkBlue!75, line width=0.5pt] (1.09, 0) rectangle (1.41, 2.04);
      \draw[fill=darkBlue!75, draw=darkBlue!75, line width=0.5pt] (1.51, 0) rectangle (1.83, 2.19);
      \node[font=\tiny, black, below=1pt] at (1.46, 0) {250};

      \draw[fill=medBlue!20, draw=darkBlue!75, line width=0.5pt] (2.18, 0) rectangle (2.50, 4.04);
      \draw[fill=darkBlue!75, draw=darkBlue!75, line width=0.5pt] (2.60, 0) rectangle (2.92, 4.51);
      \node[font=\tiny, black, below=1pt] at (2.55, 0) {500};

      \draw[->, >=stealth, black, line width=0.6pt] (-0.1,0) -- (3.12,0);
    \end{tikzpicture}
    \caption{No reuse, representative}
  \end{subfigure}
  \par\medskip
  \begin{subfigure}[b]{4.60cm}
    \centering
    \begin{tikzpicture}[x=1cm, y=1cm]

      \node[font=\tiny, black, left] at (-0.1,0) {0};
      \node[font=\tiny, black, left] at (-0.1,1.25) {75};
      \node[font=\tiny, black, left] at (-0.1,2.50) {150};
      \node[font=\tiny, black, left] at (-0.1,3.75) {225};
      \node[font=\tiny, black, left] at (-0.1,5.00) {300};

      \draw[->, >=stealth, black, line width=0.6pt] (0,-0.1) -- (0,5.40)
        node[above, font=\tiny, black] {Time (s)};

      \draw[fill=medBlue!20, draw=darkBlue!75, line width=0.5pt] (0.00, 0) rectangle (0.32, 1.81);
      \draw[fill=darkBlue!75, draw=darkBlue!75, line width=0.5pt] (0.42, 0) rectangle (0.74, 2.26);
      \node[font=\tiny, black, below=1pt] at (0.37, 0) {100};

      \draw[fill=medBlue!20, draw=darkBlue!75, line width=0.5pt] (1.09, 0) rectangle (1.41, 1.76);
      \draw[fill=darkBlue!75, draw=darkBlue!75, line width=0.5pt] (1.51, 0) rectangle (1.83, 1.98);
      \node[font=\tiny, black, below=1pt] at (1.46, 0) {250};

      \draw[fill=medBlue!20, draw=darkBlue!75, line width=0.5pt] (2.18, 0) rectangle (2.50, 2.38);
      \draw[fill=darkBlue!75, draw=darkBlue!75, line width=0.5pt] (2.60, 0) rectangle (2.92, 2.60);
      \node[font=\tiny, black, below=1pt] at (2.55, 0) {500};

      \draw[->, >=stealth, black, line width=0.6pt] (-0.1,0) -- (3.12,0);
    \end{tikzpicture}
    \caption{Reuse, midway}
  \end{subfigure}
  \hspace{0.5cm}
  \begin{subfigure}[b]{4.60cm}
    \centering
    \begin{tikzpicture}[x=1cm, y=1cm]

      \node[font=\tiny, black, left] at (-0.1,0) {0};
      \node[font=\tiny, black, left] at (-0.1,1.25) {75};
      \node[font=\tiny, black, left] at (-0.1,2.50) {150};
      \node[font=\tiny, black, left] at (-0.1,3.75) {225};
      \node[font=\tiny, black, left] at (-0.1,5.00) {300};

      \draw[->, >=stealth, black, line width=0.6pt] (0,-0.1) -- (0,5.40)
        node[above, font=\tiny, black] {Time (s)};

      \draw[fill=medBlue!20, draw=darkBlue!75, line width=0.5pt] (0.00, 0) rectangle (0.32, 1.70);
      \draw[fill=darkBlue!75, draw=darkBlue!75, line width=0.5pt] (0.42, 0) rectangle (0.74, 2.34);
      \node[font=\tiny, black, below=1pt] at (0.37, 0) {100};

      \draw[fill=medBlue!20, draw=darkBlue!75, line width=0.5pt] (1.09, 0) rectangle (1.41, 1.74);
      \draw[fill=darkBlue!75, draw=darkBlue!75, line width=0.5pt] (1.51, 0) rectangle (1.83, 2.07);
      \node[font=\tiny, black, below=1pt] at (1.46, 0) {250};

      \draw[fill=medBlue!20, draw=darkBlue!75, line width=0.5pt] (2.18, 0) rectangle (2.50, 3.28);
      \draw[fill=darkBlue!75, draw=darkBlue!75, line width=0.5pt] (2.60, 0) rectangle (2.92, 3.59);
      \node[font=\tiny, black, below=1pt] at (2.55, 0) {500};

      \draw[->, >=stealth, black, line width=0.6pt] (-0.1,0) -- (3.12,0);
    \end{tikzpicture}
    \caption{Reuse, representative}
  \end{subfigure}
  \par\medskip
  \begin{tikzpicture}[x=1cm, y=1cm]
    \draw[fill=medBlue!20, draw=darkBlue!75, line width=0.5pt] (0,0) rectangle (0.32,0.20);
    \node[font=\tiny, black, right] at (0.32,0.10) {Merge};
    \draw[fill=darkBlue!75, draw=darkBlue!75, line width=0.5pt] (2.20,0) rectangle (2.52,0.20);
    \node[font=\tiny, black, right] at (2.52,0.10) {No merge};
  \end{tikzpicture}
  \caption{Root node computation time of adaptive partitioning per configuration for the B-MPCVRP. Results are averaged over all instances with $n=25$, $T=4$, and $\delta=10\%$.}
  \label{fig:cvrp_root_strategies}
\end{figure}
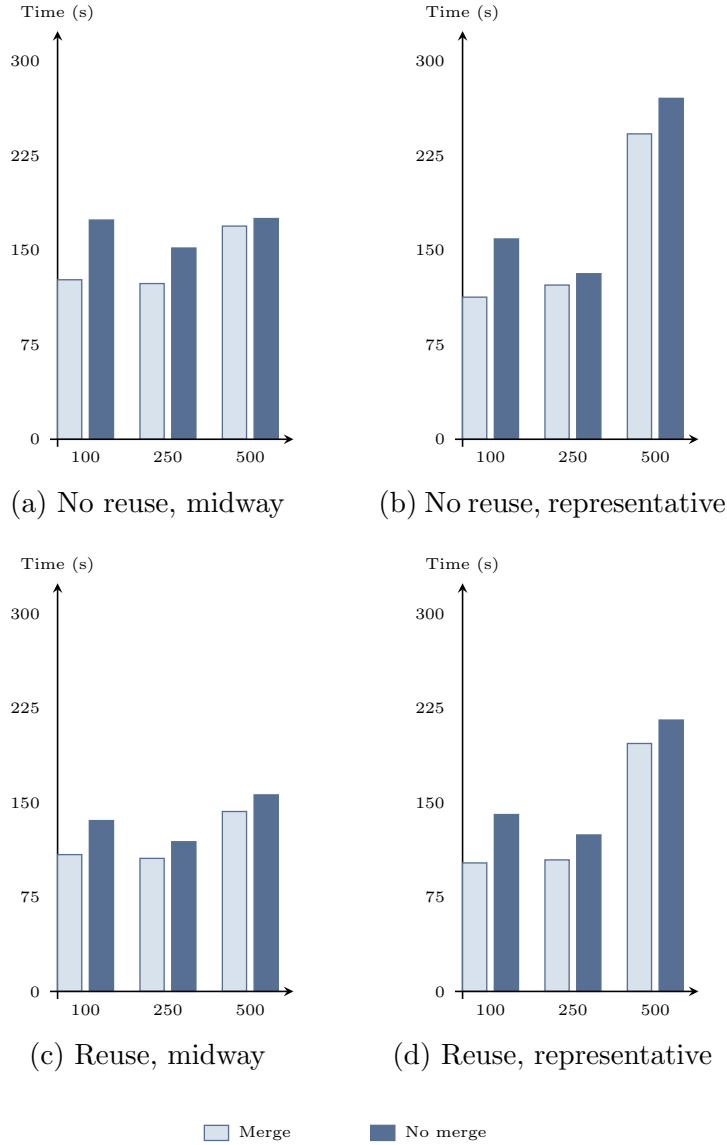

Figure~\ref{fig:cvrp_root_strategies} plots the average computation times for all 24 configurations. Detailed results are reported in Table~\ref{tab:root_strategies_cvrp} in Appendix~\ref{app:results}. First, merging has a consistent positive effect across all configurations, by keeping the bucket graph compact to compute representatives more efficiently. Reusing representatives from the previous column generation iteration is similarly beneficial, again by avoiding representative computations when dual variables have changed only moderately. Moreover, an intermediate initialization performs best overall: an initial bucket width of 100 requires finding representatives for many small buckets, while a width of 500 leads to slower convergence as more refinement iterations are needed to align the pessimistic and optimistic bounds. In fact, Table~\ref{tab:root_strategies_cvrp} shows that finding representatives consumes over 99\% of computation time among all configurations. Finally, the midpoint refinement strategy outperforms the representative strategy, especially at larger initial bucket widths. Altogether, the best-performing configuration combines an initial bucket width of 250 with merging, midpoint refinement, and reuse of representatives. We adopt this configuration for the remainder of our experiments in this section.

\subsection{Adaptive Partitioning versus Enumerative Benchmark}
\label{subsec:cvrp_benchmark}

Figure~\ref{fig:cvrp_root_comparison} shows the average root node computation times for the adaptive partitioning algorithm and the enumerative benchmark across all instance configurations. Detailed iteration counts are provided in Table~\ref{tab:cvrp_root_comparison} in Appendix~\ref{app:results}. The main observation is that the adaptive partitioning algorithm outperforms the benchmark on all but the smallest instances ($n = 15$), and its advantage grows sharply with the number of customers. In the largest instances with $n = 25$ and $T = 4$, it achieves a speed-up of a factor of over 13 (875 vs. 66 seconds). These benefits are largely driven by a reduction in the cost per pricing iteration: the bucket graph keeps pricing tractable as instance size grows, whereas enumerating non-dominated routes becomes increasingly expensive.

\begin{figure}[h!]
  \centering
  \begin{subfigure}[b]{4.30cm}
    \centering
    \begin{tikzpicture}[x=1cm, y=1cm]

      \node[font=\tiny, black, left] at (-0.1,0) {0};
      \node[font=\tiny, black, left] at (-0.1,1.25) {1};
      \node[font=\tiny, black, left] at (-0.1,2.50) {2};
      \node[font=\tiny, black, left] at (-0.1,3.75) {3};
      \node[font=\tiny, black, left] at (-0.1,5.00) {4};

      \draw[->, >=stealth, black, line width=0.6pt] (0,-0.1) -- (0,5.40)
        node[above, font=\tiny, black] {Time (s)};

      \draw[fill=medBlue!20, draw=darkBlue!75, line width=0.5pt] (0.00, 0) rectangle (0.32, 1.31);
      \draw[fill=darkBlue!75, draw=darkBlue!75, line width=0.5pt] (0.44, 0) rectangle (0.76, 0.64);
      \node[font=\tiny, black, below=1pt] at (0.38, 0) {$T\!=\!2$};

      \draw[fill=medBlue!20, draw=darkBlue!75, line width=0.5pt] (1.16, 0) rectangle (1.48, 2.77);
      \draw[fill=darkBlue!75, draw=darkBlue!75, line width=0.5pt] (1.60, 0) rectangle (1.92, 1.08);
      \node[font=\tiny, black, below=1pt] at (1.54, 0) {$T\!=\!3$};

      \draw[fill=medBlue!20, draw=darkBlue!75, line width=0.5pt] (2.32, 0) rectangle (2.64, 4.77);
      \draw[fill=darkBlue!75, draw=darkBlue!75, line width=0.5pt] (2.76, 0) rectangle (3.08, 2.11);
      \node[font=\tiny, black, below=1pt] at (2.70, 0) {$T\!=\!4$};

      \draw[->, >=stealth, black, line width=0.6pt] (-0.1,0) -- (3.28,0);
    \end{tikzpicture}
    \caption{$n = 15$}
  \end{subfigure}
  \hspace{0.5cm}
  \begin{subfigure}[b]{4.30cm}
    \centering
    \begin{tikzpicture}[x=1cm, y=1cm]

      \node[font=\tiny, black, left] at (-0.1,0) {0};
      \node[font=\tiny, black, left] at (-0.1,1.25) {15};
      \node[font=\tiny, black, left] at (-0.1,2.50) {30};
      \node[font=\tiny, black, left] at (-0.1,3.75) {45};
      \node[font=\tiny, black, left] at (-0.1,5.00) {60};

      \draw[->, >=stealth, black, line width=0.6pt] (0,-0.1) -- (0,5.40)
        node[above, font=\tiny, black] {Time (s)};

      \draw[fill=medBlue!20, draw=darkBlue!75, line width=0.5pt] (0.00, 0) rectangle (0.32, 0.40);
      \draw[fill=darkBlue!75, draw=darkBlue!75, line width=0.5pt] (0.44, 0) rectangle (0.76, 0.75);
      \node[font=\tiny, black, below=1pt] at (0.38, 0) {$T\!=\!2$};

      \draw[fill=medBlue!20, draw=darkBlue!75, line width=0.5pt] (1.16, 0) rectangle (1.48, 0.56);
      \draw[fill=darkBlue!75, draw=darkBlue!75, line width=0.5pt] (1.60, 0) rectangle (1.92, 1.31);
      \node[font=\tiny, black, below=1pt] at (1.54, 0) {$T\!=\!3$};

      \draw[fill=medBlue!20, draw=darkBlue!75, line width=0.5pt] (2.32, 0) rectangle (2.64, 1.24);
      \draw[fill=darkBlue!75, draw=darkBlue!75, line width=0.5pt] (2.76, 0) rectangle (3.08, 4.32);
      \node[font=\tiny, black, below=1pt] at (2.70, 0) {$T\!=\!4$};

      \draw[->, >=stealth, black, line width=0.6pt] (-0.1,0) -- (3.28,0);
    \end{tikzpicture}
    \caption{$n = 20$}
  \end{subfigure}
  \hspace{0.5cm}
  \begin{subfigure}[b]{4.30cm}
    \centering
    \begin{tikzpicture}[x=1cm, y=1cm]

      \node[font=\tiny, black, left] at (-0.1,0) {0};
      \node[font=\tiny, black, left] at (-0.1,1.25) {250};
      \node[font=\tiny, black, left] at (-0.1,2.50) {500};
      \node[font=\tiny, black, left] at (-0.1,3.75) {750};
      \node[font=\tiny, black, left] at (-0.1,5.00) {1{,}000};

      \draw[->, >=stealth, black, line width=0.6pt] (0,-0.1) -- (0,5.40)
        node[above, font=\tiny, black] {Time (s)};

      \draw[fill=medBlue!20, draw=darkBlue!75, line width=0.5pt] (0.00, 0) rectangle (0.32, 0.08);
      \draw[fill=darkBlue!75, draw=darkBlue!75, line width=0.5pt] (0.44, 0) rectangle (0.76, 0.68);
      \node[font=\tiny, black, below=1pt] at (0.38, 0) {$T\!=\!2$};

      \draw[fill=medBlue!20, draw=darkBlue!75, line width=0.5pt] (1.16, 0) rectangle (1.48, 0.17);
      \draw[fill=darkBlue!75, draw=darkBlue!75, line width=0.5pt] (1.60, 0) rectangle (1.92, 1.01);
      \node[font=\tiny, black, below=1pt] at (1.54, 0) {$T\!=\!3$};

      \draw[fill=medBlue!20, draw=darkBlue!75, line width=0.5pt] (2.32, 0) rectangle (2.64, 0.33);
      \draw[fill=darkBlue!75, draw=darkBlue!75, line width=0.5pt] (2.76, 0) rectangle (3.08, 4.38);
      \node[font=\tiny, black, below=1pt] at (2.70, 0) {$T\!=\!4$};

      \draw[->, >=stealth, black, line width=0.6pt] (-0.1,0) -- (3.28,0);
    \end{tikzpicture}
    \caption{$n = 25$}
  \end{subfigure}
  \par\medskip
  \begin{tikzpicture}[x=1cm, y=1cm]
    \draw[fill=medBlue!20, draw=darkBlue!75, line width=0.5pt] (0,0) rectangle (0.32,0.20);
    \node[font=\tiny, black, right] at (0.32,0.10) {Adaptive};
    \draw[fill=darkBlue!75, draw=darkBlue!75, line width=0.5pt] (2.80,0) rectangle (3.12,0.20);
    \node[font=\tiny, black, right] at (3.12,0.10) {Enumerative};
  \end{tikzpicture}
  \caption{Root node computation time for the B-MPCVRP: adaptive partitioning and subpath-based enumerative benchmark. Results are averaged over all five instances and values of $\delta$.}
  \label{fig:cvrp_root_comparison}
\end{figure}
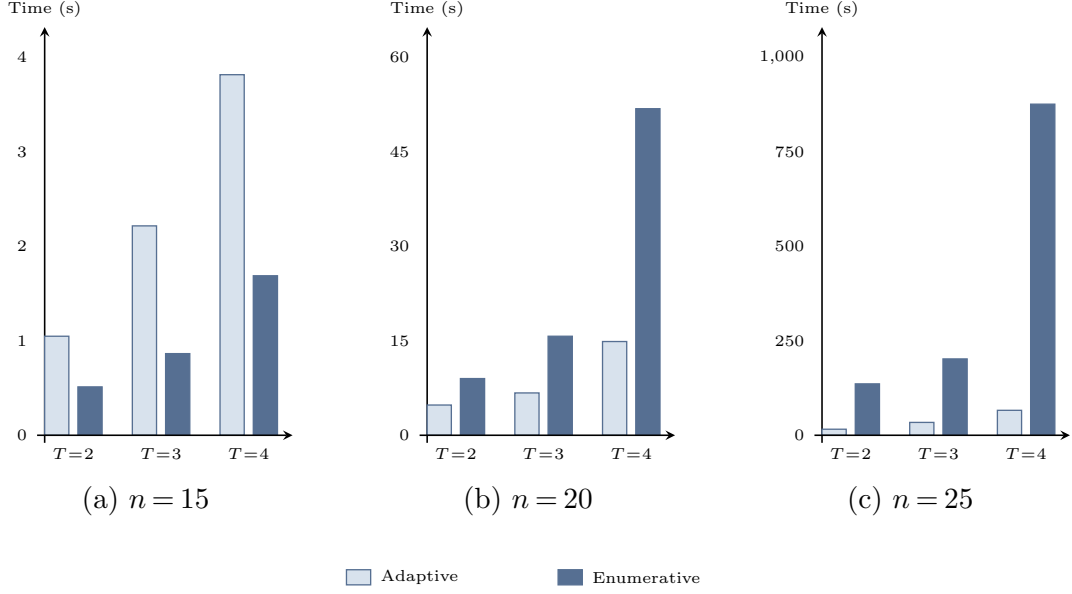

On the other hand, the adaptive partitioning algorithm requires substantially more iterations to prove optimality. Since the bucket graph is a coarsened approximation of the subpath graph, it is not always able to directly find the overall most negative reduced cost column. This explains why the benchmark is roughly twice as fast for $n = 15$, where enumeration is still relatively cheap. Yet, the adaptive partitioning algorithm developed in this paper yields significant improvement in computational times and overall scalability, thereby substantially expanding the range of instance sizes that can be solved within manageable computation times.

\subsection{Branch-Price-and-Cut}
\label{subsec:cvrp_branchPrice}

We now compare branch-price-and-cut for the path-based (schedule-based) formulation against its counterpart for a subpath-based (route-based) formulation. The baseline contains variables for all possible subpaths (routes) instead of all possible paths (route schedules). Let $R_{k}^{t}$ denote the set of feasible routes that can be assigned to vehicle $k \in K$ on day $t \in T$. For route $r \in R_{k}^{t}$, let $d_r$ denote its distance and let binary parameter $a_{rn}$ denote whether customer $n \in N_t$ is covered by route $r$. The route formulation reads
\begin{subequations}
\begin{align}
\min \quad & \sum_{t \in T} \sum_{k \in K} \sum_{r \in R^t_k} d_r x_{r} & \\ 
\text{s.t.} \quad & \sum_{k \in K} \sum_{r \in R^t_k} a_{rn} x_r = 1 && \forall t \in T, n \in N_t \\ 
& \sum_{r \in R^{t}_{k}} x_r = 1 && \forall t \in T, k \in K \\ 
& \sum_{t \in T} \sum_{r \in R^{t}_{k}} d_r x_r \leq D && \forall k \in K \label{eq:cvrp_subpath_distance} \\
& x_r \in \{0, 1\} && \forall t \in T, k \in K, r\in R^t_k.
\end{align}
\label{eq:cvrp_subpath_ip}%
\end{subequations} 
Whereas the schedule-based formulation relegates the distance constraint to the pricing problem, the route-based formulation contains explicit knapsack constraints enforcing the maximum distance along each path~\eqref{eq:cvrp_subpath_distance}. The LP relaxation can be solved via standard column generation, with a decomposable pricing problem across vehicle-day pairs $(k, t)$. As these subproblems share the same structure as the labeling algorithm used to determine bucket representatives, they can be solved using a version of that algorithm where no distance upper bound is imposed.

Apart from the pricing algorithms, we use nearly identical branch-price-and-cut algorithms for both formulations. We strengthen the LP relaxation by separating rounded capacity inequalities using the heuristic routines of~\citet{lysgaard2004new}. Branching follows a hierarchical strategy: for the route formulation we branch on (i) customer-to-vehicle assignments, (ii) arc cutsets, and (iii) individual arcs. For the schedule-based formulation we branch on (ii) and (iii) only. While this does not guarantee integrality of the resulting solution, feasibility of an integral completion can be verified efficiently by solving a small mixed-integer program. To find primal solutions, we invoke a restricted master heuristic every ten nodes with a time limit of five seconds. Since primal solutions are difficult to obtain within the schedule-based formulation, we decompose all schedules into their daily routes and apply the restricted master heuristic to the route formulation.

\begin{table}[htbp!]
  \centering
  \caption{Branch-price-and-cut for the B-MPCVRP: schedule- and route-based formulation. Results are averaged over five instances per configuration.}
  \label{tab:cvrp_branching}
  \begin{adjustbox}{max width=\linewidth, max totalheight=0.95\textheight, keepaspectratio}
  \begin{tabular}{rrr rrrr rrrr}
    \toprule
    & & & \multicolumn{4}{c}{Schedule} & \multicolumn{4}{c}{Route} \\
    \cmidrule(lr){4-7} \cmidrule(lr){8-11}
    $n$ & $T$ & $\delta$ (\%) & Opt. & Inf. & Time (s) & Nodes & Opt. & Inf. & Time (s) & Nodes \\
    \midrule
    \multirow{16}{*}{15} & \multirow{5}{*}{2} & 10 & 5/5 & 5/5 & 2 & 1 & 4/5 & 4/5 & 1{,}057 & 1{,}191 \\
     &  & 30 & 5/5 & 3/5 & 5 & 8 & 2/5 & 0/5 & 2{,}347 & 2{,}066 \\
     &  & 50 & 5/5 & 1/5 & 101 & 165 & 2/5 & 0/5 & 2{,}168 & 1{,}495 \\
     &  & 70 & 5/5 & 0/5 & 30 & 84 & 5/5 & 0/5 & 411 & 679 \\
     &  & 90 & 5/5 & 0/5 & 17 & 73 & 4/5 & 0/5 & 741 & 778 \\
    \cmidrule(lr){2-11}
     & \multirow{5}{*}{3} & 10 & 5/5 & 5/5 & 8 & 1 & 3/5 & 3/5 & 1{,}837 & 1{,}999 \\
     &  & 30 & 4/5 & 4/5 & 735 & 174 & 1/5 & 1/5 & 3{,}598 & 2{,}909 \\
     &  & 50 & 5/5 & 3/5 & 215 & 89 & 0/5 & 0/5 & 3{,}600 & 2{,}708 \\
     &  & 70 & 4/5 & 1/5 & 812 & 247 & 0/5 & 0/5 & 3{,}600 & 2{,}792 \\
     &  & 90 & 5/5 & 0/5 & 131 & 94 & 1/5 & 0/5 & 3{,}502 & 3{,}170 \\
    \cmidrule(lr){2-11}
     & \multirow{5}{*}{4} & 10 & 4/5 & 4/5 & 756 & 10 & 0/5 & 0/5 & 3{,}600 & 3{,}271 \\
     &  & 30 & 3/5 & 2/5 & 1{,}731 & 64 & 0/5 & 0/5 & 3{,}600 & 2{,}954 \\
     &  & 50 & 3/5 & 0/5 & 1{,}850 & 90 & 0/5 & 0/5 & 3{,}600 & 3{,}228 \\
     &  & 70 & 4/5 & 0/5 & 916 & 80 & 0/5 & 0/5 & 3{,}600 & 3{,}268 \\
     &  & 90 & 5/5 & 0/5 & 284 & 47 & 0/5 & 0/5 & 3{,}600 & 3{,}249 \\
    \cmidrule(lr){2-11}
     & & \textbf{Avg.} & \textbf{67/75} & \textbf{28/75} & \textbf{506} & \textbf{82} & \textbf{22/75} & \textbf{8/75} & \textbf{2{,}724} & \textbf{2{,}384} \\
    \midrule
    \multirow{16}{*}{20} & \multirow{5}{*}{2} & 10 & 5/5 & 5/5 & 9 & 1 & 5/5 & 5/5 & 1{,}135 & 1{,}037 \\
     &  & 30 & 5/5 & 4/5 & 13 & 3 & 0/5 & 0/5 & 3{,}600 & 1{,}831 \\
     &  & 50 & 4/5 & 2/5 & 732 & 109 & 1/5 & 0/5 & 3{,}194 & 1{,}781 \\
     &  & 70 & 5/5 & 1/5 & 255 & 62 & 4/5 & 0/5 & 1{,}643 & 1{,}179 \\
     &  & 90 & 4/5 & 0/5 & 735 & 115 & 3/5 & 0/5 & 1{,}571 & 949 \\
    \cmidrule(lr){2-11}
     & \multirow{5}{*}{3} & 10 & 5/5 & 5/5 & 73 & 5 & 2/5 & 2/5 & 3{,}269 & 2{,}021 \\
     &  & 30 & 4/5 & 3/5 & 983 & 52 & 0/5 & 0/5 & 3{,}600 & 1{,}809 \\
     &  & 50 & 2/5 & 1/5 & 2{,}358 & 144 & 0/5 & 0/5 & 3{,}600 & 1{,}961 \\
     &  & 70 & 3/5 & 1/5 & 1{,}632 & 113 & 1/5 & 0/5 & 3{,}409 & 1{,}894 \\
     &  & 90 & 4/5 & 1/5 & 854 & 66 & 0/5 & 0/5 & 3{,}600 & 1{,}975 \\
    \cmidrule(lr){2-11}
     & \multirow{5}{*}{4} & 10 & 3/5 & 3/5 & 1{,}591 & 8 & 0/5 & 0/5 & 3{,}600 & 2{,}258 \\
     &  & 30 & 2/5 & 2/5 & 2{,}655 & 16 & 0/5 & 0/5 & 3{,}600 & 2{,}083 \\
     &  & 50 & 0/5 & 0/5 & 3{,}600 & 37 & 0/5 & 0/5 & 3{,}600 & 2{,}099 \\
     &  & 70 & 2/5 & 0/5 & 2{,}856 & 32 & 0/5 & 0/5 & 3{,}600 & 2{,}068 \\
     &  & 90 & 1/5 & 0/5 & 3{,}032 & 49 & 0/5 & 0/5 & 3{,}600 & 2{,}078 \\
    \cmidrule(lr){2-11}
     & & \textbf{Avg.} & \textbf{49/75} & \textbf{28/75} & \textbf{1{,}425} & \textbf{54} & \textbf{16/75} & \textbf{7/75} & \textbf{3{,}108} & \textbf{1{,}802} \\
    \midrule
    \multirow{16}{*}{25} & \multirow{5}{*}{2} & 10 & 4/5 & 4/5 & 749 & 43 & 3/5 & 3/5 & 1{,}965 & 627 \\
     &  & 30 & 3/5 & 3/5 & 1{,}462 & 139 & 0/5 & 0/5 & 3{,}600 & 1{,}036 \\
     &  & 50 & 4/5 & 2/5 & 865 & 83 & 2/5 & 0/5 & 2{,}423 & 718 \\
     &  & 70 & 3/5 & 0/5 & 1{,}910 & 172 & 2/5 & 0/5 & 2{,}411 & 726 \\
     &  & 90 & 4/5 & 0/5 & 868 & 71 & 4/5 & 0/5 & 1{,}429 & 496 \\
    \cmidrule(lr){2-11}
     & \multirow{5}{*}{3} & 10 & 2/5 & 2/5 & 2{,}223 & 36 & 0/5 & 0/5 & 3{,}600 & 844 \\
     &  & 30 & 0/5 & 0/5 & 3{,}600 & 85 & 0/5 & 0/5 & 3{,}600 & 796 \\
     &  & 50 & 1/5 & 0/5 & 3{,}052 & 80 & 0/5 & 0/5 & 3{,}600 & 861 \\
     &  & 70 & 1/5 & 0/5 & 2{,}935 & 73 & 0/5 & 0/5 & 3{,}600 & 918 \\
     &  & 90 & 1/5 & 0/5 & 2{,}903 & 79 & 0/5 & 0/5 & 3{,}600 & 941 \\
    \cmidrule(lr){2-11}
     & \multirow{5}{*}{4} & 10 & 0/5 & 0/5 & 3{,}600 & 19 & 0/5 & 0/5 & 3{,}600 & 801 \\
     &  & 30 & 1/5 & 0/5 & 3{,}292 & 25 & 0/5 & 0/5 & 3{,}600 & 808 \\
     &  & 50 & 1/5 & 0/5 & 2{,}998 & 25 & 0/5 & 0/5 & 3{,}600 & 793 \\
     &  & 70 & 1/5 & 0/5 & 3{,}260 & 31 & 0/5 & 0/5 & 3{,}600 & 813 \\
     &  & 90 & 1/5 & 0/5 & 3{,}003 & 31 & 0/5 & 0/5 & 3{,}600 & 821 \\
    \cmidrule(lr){2-11}
     & & \textbf{Avg.} & \textbf{27/75} & \textbf{11/75} & \textbf{2{,}448} & \textbf{66} & \textbf{11/75} & \textbf{3/75} & \textbf{3{,}189} & \textbf{800} \\
    \bottomrule
  \end{tabular}
  \end{adjustbox}
\end{table}

Table~\ref{tab:cvrp_branching} reports the performance of both branch-price-and-cut algorithms across all configurations, subject to a time limit of one hour. For every configuration, we report the number of instances solved to optimality or proven infeasible, computation time, and number of explored nodes, averaged over five instances. The main observation is that the path-based (schedule-based) formulation consistently outperforms its subpath-based (route-based) counterpart, solving 67, 49, and 27 out of 75 instances to optimality for $n = 15, 20, 25$ respectively, compared to only 22, 16, and 11. This advantage stems in large part from the tighter LP relaxation obtained by embedding the knapsack constraint within the pricing problem: the schedule-based formulation explores far fewer branch-and-bound nodes on average and frequently concludes infeasibility already at the root node. The route-based formulation, in contrast, struggles to close the optimality gap. As expected,  the schedule-based formulation involves a longer time per node, reflecting the greater cost of path-based pricing relative to that of subpath-based pricing. Yet, the path-based (schedule-based) formulation yields net scalability benefits enabled by the adaptive partitioning algorithm.

\section{Computational Results on the RCSP} 
\label{sec:crew}

We now analyze the performance of the adaptive partitioning algorithm on the RCSP (Example~\ref{ex:rcsp}). We consider a template-based formulation as in problem~\eqref{eq:master_ip}, and solve its LP relaxation via column generation with the adaptive partitioning algorithm. Buckets are defined along two resource dimensions, specifying allowed start and end time intervals for duties. The problem of finding a bucket representative corresponds to the pricing problem of a regular crew scheduling problem and can be solved in polynomial time \Citep{rossum2022fast}.

We describe the experimental set-up in Section~\ref{subsec:crew_setup}. As in the B-MPCVRP, we analyze the effect of design choices on the performance of the bucket-based algorithm in Section~\ref{subsec:crew_strategies} and compare the best configuration with the benchmark in Section~\ref{subsec:crew_benchmark}. Finally, we employ the adaptive partitioning algorithm within a diving heuristic to obtain integral solutions in Section~\ref{subsec:crew_diving}.

\subsection{Set-Up}
\label{subsec:crew_setup}

Following \Citet{rossum2025benders}, we use crew scheduling data from Netherlands Railways, drawing on historical planning data from three weeks in October 2021. We construct instances for three crew bases: the small crew base Den Bosch (Ht) and the medium-sized crew bases Nijmegen (Nm) and Amersfoort (Amf). For each crew base, we construct one instance per day of the week, each containing three historical weeks as timetable scenarios. Table~\ref{tab:rcsp_instances} reports the average number of tasks and connections between tasks per crew base and day of the week. The largest instance contains $3 \times 593$ tasks to be covered across all scenarios. For each instance, we consider three settings for the maximum template length: 9h, 9h15, and 9h30, yielding a total of $63$ instances.

\begin{table}[htbp!]
  \centering
  \caption{Average number of tasks and connections per crew base and day of week, averaged over three weeks.}
  \label{tab:rcsp_instances}
  \begin{tabular}{l rr rr rr}
    \toprule
     & \multicolumn{2}{c}{Ht} & \multicolumn{2}{c}{Nm} & \multicolumn{2}{c}{Amf} \\
    \cmidrule(lr){2-3} \cmidrule(lr){4-5} \cmidrule(lr){6-7}
    Day & Tasks & Arcs & Tasks & Arcs & Tasks & Arcs \\
    \midrule
    Mon & 266 & 2{,}756 & 483 & 6{,}359 & 551 & 7{,}636 \\
    Tue & 291 & 3{,}255 & 491 & 6{,}344 & 552 & 7{,}804 \\
    Wed & 285 & 3{,}312 & 473 & 6{,}533 & 535 & 7{,}437 \\
    Thu & 334 & 3{,}755 & 516 & 7{,}129 & 593 & 7{,}574 \\
    Fri & 329 & 3{,}757 & 563 & 7{,}469 & 580 & 7{,}392 \\
    Sat & 247 & 2{,}575 & 437 & 5{,}546 & 440 & 4{,}510 \\
    Sun & 213 & 2{,}241 & 345 & 4{,}337 & 405 & 4{,}437 \\
    \bottomrule
  \end{tabular}
\end{table}

The algorithmic settings are identical to those described in Section~\ref{subsec:cvrp_setup}, with two exceptions. First, rather than adding all generated columns to the RMP in each pricing iteration, we retain only those columns that are sufficiently task-disjoint from one another, following the approach of \citet{breugem2022column}. Second, due to the large number of cover constraints and tasks per duty, we use a self-tuning dual smoothing parameter.

\subsection{Configuration of Adaptive Partitioning}
\label{subsec:crew_strategies}

Figure~\ref{fig:csp_root_strategies} plots the average computation time for different configurations of the adaptive partitioning algorithm when solving the root node of instances from crew base Ht with a maximum template length of 9 hours, averaged over all days of the week. We consider initial interval widths of 20, 30, and 40 minutes. A detailed breakdown of computation time across subroutines is provided in Table~\ref{tab:root_strategies_csp} in Appendix~\ref{app:results}. The overall patterns differ from those observed for the B-MPCVRP, which can be attributed to the different computational profile of this application: finding representatives is relatively cheap, so the majority of computation time is consumed by pessimistic pricing.

\begin{figure}[h!]
  \centering
  \begin{subfigure}[b]{4.60cm}
    \centering
    \begin{tikzpicture}[x=1cm, y=1cm]

      \node[font=\tiny, black, left] at (-0.1,0) {0};
      \node[font=\tiny, black, left] at (-0.1,1.25) {175};
      \node[font=\tiny, black, left] at (-0.1,2.50) {350};
      \node[font=\tiny, black, left] at (-0.1,3.75) {525};
      \node[font=\tiny, black, left] at (-0.1,5.00) {700};

      \draw[->, >=stealth, black, line width=0.6pt] (0,-0.1) -- (0,5.40)
        node[above, font=\tiny, black] {Time (s)};

      \draw[fill=medBlue!20, draw=darkBlue!75, line width=0.5pt] (0.00, 0) rectangle (0.32, 1.94);
      \draw[fill=darkBlue!75, draw=darkBlue!75, line width=0.5pt] (0.42, 0) rectangle (0.74, 5.00);
      \node[font=\tiny, black, below=1pt] at (0.37, 0) {20};

      \draw[fill=medBlue!20, draw=darkBlue!75, line width=0.5pt] (1.09, 0) rectangle (1.41, 1.66);
      \draw[fill=darkBlue!75, draw=darkBlue!75, line width=0.5pt] (1.51, 0) rectangle (1.83, 3.45);
      \node[font=\tiny, black, below=1pt] at (1.46, 0) {30};

      \draw[fill=medBlue!20, draw=darkBlue!75, line width=0.5pt] (2.18, 0) rectangle (2.50, 2.70);
      \draw[fill=darkBlue!75, draw=darkBlue!75, line width=0.5pt] (2.60, 0) rectangle (2.92, 2.68);
      \node[font=\tiny, black, below=1pt] at (2.55, 0) {40};

      \draw[->, >=stealth, black, line width=0.6pt] (-0.1,0) -- (3.12,0);
    \end{tikzpicture}
    \caption{No reuse, midway}
  \end{subfigure}
  \hspace{0.5cm}
  \begin{subfigure}[b]{4.60cm}
    \centering
    \begin{tikzpicture}[x=1cm, y=1cm]

      \node[font=\tiny, black, left] at (-0.1,0) {0};
      \node[font=\tiny, black, left] at (-0.1,1.25) {175};
      \node[font=\tiny, black, left] at (-0.1,2.50) {350};
      \node[font=\tiny, black, left] at (-0.1,3.75) {525};
      \node[font=\tiny, black, left] at (-0.1,5.00) {700};

      \draw[->, >=stealth, black, line width=0.6pt] (0,-0.1) -- (0,5.40)
        node[above, font=\tiny, black] {Time (s)};

      \draw[fill=medBlue!20, draw=darkBlue!75, line width=0.5pt] (0.00, 0) rectangle (0.32, 1.52);
      \draw[fill=darkBlue!75, draw=darkBlue!75, line width=0.5pt] (0.42, 0) rectangle (0.74, 5.00);
      \node[font=\tiny, black, below=1pt] at (0.37, 0) {20};

      \draw[fill=medBlue!20, draw=darkBlue!75, line width=0.5pt] (1.09, 0) rectangle (1.41, 1.79);
      \draw[fill=darkBlue!75, draw=darkBlue!75, line width=0.5pt] (1.51, 0) rectangle (1.83, 3.20);
      \node[font=\tiny, black, below=1pt] at (1.46, 0) {30};

      \draw[fill=medBlue!20, draw=darkBlue!75, line width=0.5pt] (2.18, 0) rectangle (2.50, 2.75);
      \draw[fill=darkBlue!75, draw=darkBlue!75, line width=0.5pt] (2.60, 0) rectangle (2.92, 3.21);
      \node[font=\tiny, black, below=1pt] at (2.55, 0) {40};

      \draw[->, >=stealth, black, line width=0.6pt] (-0.1,0) -- (3.12,0);
    \end{tikzpicture}
    \caption{No reuse, representative}
  \end{subfigure}
  \par\medskip
  \begin{subfigure}[b]{4.60cm}
    \centering
    \begin{tikzpicture}[x=1cm, y=1cm]

      \node[font=\tiny, black, left] at (-0.1,0) {0};
      \node[font=\tiny, black, left] at (-0.1,1.25) {625};
      \node[font=\tiny, black, left] at (-0.1,2.50) {1{,}250};
      \node[font=\tiny, black, left] at (-0.1,3.75) {1{,}875};
      \node[font=\tiny, black, left] at (-0.1,5.00) {2{,}500};

      \draw[->, >=stealth, black, line width=0.6pt] (0,-0.1) -- (0,5.40)
        node[above, font=\tiny, black] {Time (s)};

      \draw[fill=medBlue!20, draw=darkBlue!75, line width=0.5pt] (0.00, 0) rectangle (0.32, 1.53);
      \draw[fill=darkBlue!75, draw=darkBlue!75, line width=0.5pt] (0.42, 0) rectangle (0.74, 4.80);
      \node[font=\tiny, black, below=1pt] at (0.37, 0) {20};

      \draw[fill=medBlue!20, draw=darkBlue!75, line width=0.5pt] (1.09, 0) rectangle (1.41, 1.07);
      \draw[fill=darkBlue!75, draw=darkBlue!75, line width=0.5pt] (1.51, 0) rectangle (1.83, 2.64);
      \node[font=\tiny, black, below=1pt] at (1.46, 0) {30};

      \draw[fill=medBlue!20, draw=darkBlue!75, line width=0.5pt] (2.18, 0) rectangle (2.50, 2.49);
      \draw[fill=darkBlue!75, draw=darkBlue!75, line width=0.5pt] (2.60, 0) rectangle (2.92, 2.79);
      \node[font=\tiny, black, below=1pt] at (2.55, 0) {40};

      \draw[->, >=stealth, black, line width=0.6pt] (-0.1,0) -- (3.12,0);
    \end{tikzpicture}
    \caption{Reuse, midway}
  \end{subfigure}
  \hspace{0.5cm}
  \begin{subfigure}[b]{4.60cm}
    \centering
    \begin{tikzpicture}[x=1cm, y=1cm]

      \node[font=\tiny, black, left] at (-0.1,0) {0};
      \node[font=\tiny, black, left] at (-0.1,1.25) {625};
      \node[font=\tiny, black, left] at (-0.1,2.50) {1{,}250};
      \node[font=\tiny, black, left] at (-0.1,3.75) {1{,}875};
      \node[font=\tiny, black, left] at (-0.1,5.00) {2{,}500};

      \draw[->, >=stealth, black, line width=0.6pt] (0,-0.1) -- (0,5.40)
        node[above, font=\tiny, black] {Time (s)};

      \draw[fill=medBlue!20, draw=darkBlue!75, line width=0.5pt] (0.00, 0) rectangle (0.32, 1.49);
      \draw[fill=darkBlue!75, draw=darkBlue!75, line width=0.5pt] (0.42, 0) rectangle (0.74, 4.61);
      \node[font=\tiny, black, below=1pt] at (0.37, 0) {20};

      \draw[fill=medBlue!20, draw=darkBlue!75, line width=0.5pt] (1.09, 0) rectangle (1.41, 1.02);
      \draw[fill=darkBlue!75, draw=darkBlue!75, line width=0.5pt] (1.51, 0) rectangle (1.83, 2.86);
      \node[font=\tiny, black, below=1pt] at (1.46, 0) {30};

      \draw[fill=medBlue!20, draw=darkBlue!75, line width=0.5pt] (2.18, 0) rectangle (2.50, 2.97);
      \draw[fill=darkBlue!75, draw=darkBlue!75, line width=0.5pt] (2.60, 0) rectangle (2.92, 2.54);
      \node[font=\tiny, black, below=1pt] at (2.55, 0) {40};

      \draw[->, >=stealth, black, line width=0.6pt] (-0.1,0) -- (3.12,0);
    \end{tikzpicture}
    \caption{Reuse, representative}
  \end{subfigure}
  \par\medskip
  \begin{tikzpicture}[x=1cm, y=1cm]
    \draw[fill=medBlue!20, draw=darkBlue!75, line width=0.5pt] (0,0) rectangle (0.32,0.20);
    \node[font=\tiny, black, right] at (0.32,0.10) {Merge};
    \draw[fill=darkBlue!75, draw=darkBlue!75, line width=0.5pt] (2.20,0) rectangle (2.52,0.20);
    \node[font=\tiny, black, right] at (2.52,0.10) {No merge};
  \end{tikzpicture}
  \caption{Root node computation time of adaptive partitioning per configuration for the RCSP. Results for crew base Ht and a template length of 9 hours, averaged over all days of the week.}
  \label{fig:csp_root_strategies}
\end{figure}
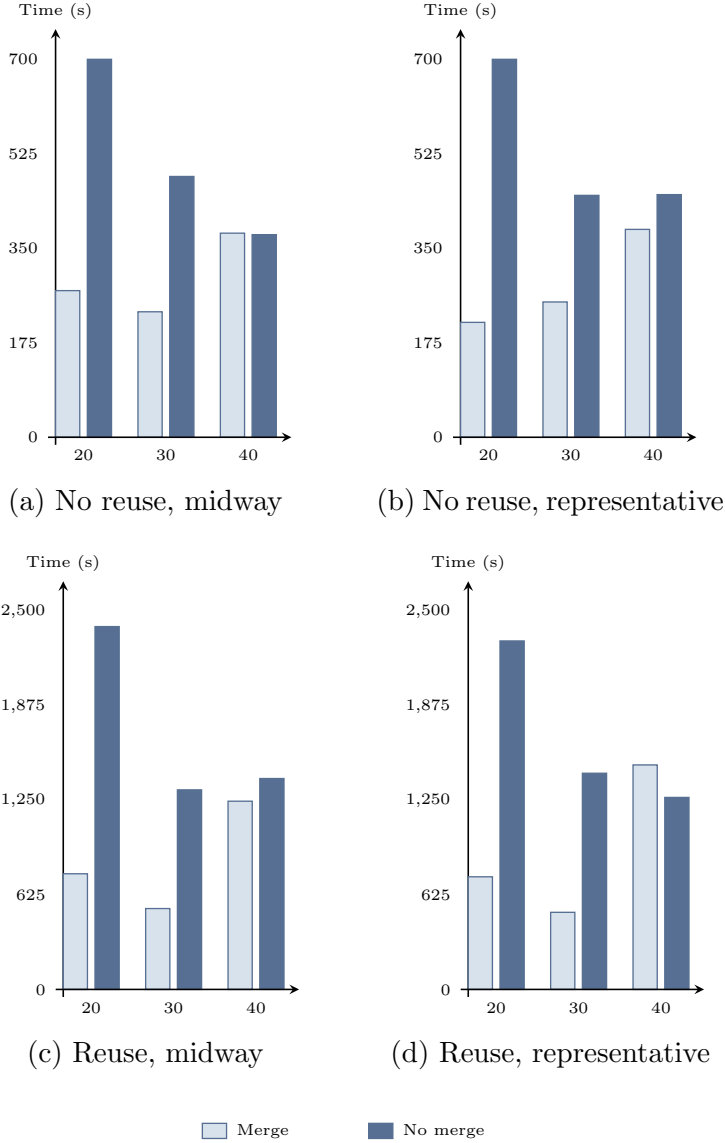

The most striking difference from the B-MPCVRP is the effect of reusing representatives. Rather than reducing computation time, reuse consistently worsens performance by a large margin (upper vs. lower rows of Figure~\ref{fig:csp_root_strategies}). When computing optimal representatives is cheap, as in the RCSP, the savings from reuse do not outweigh the slower convergence of the overall algorithm. By contrast, merging retains a consistently positive effect across all configurations, as it reduces the size of the bucket graph and thereby limits the cost of the pessimistic pricing step. The choice of the refinement strategy has little impact on computation time in this problem, though the representative strategy shows a slight edge over the midpoint strategy. Altogether, the best-performing configuration combines no reuse, representative refinement, merging, and an initial bucket width of 20 minutes, which we adopt for all remaining experiments in this section.

\subsection{Adaptive Partitioning versus Enumerative Benchmark}
\label{subsec:crew_benchmark}

We now compare the adaptive partitioning algorithm against the enumerative benchmark. Figure~\ref{fig:rcsp_root_comparison} shows the root node computation times of both methods, averaged over all days of the week, with iteration counts reported in Table~\ref{tab:rcsp_root_comparison} in Appendix~\ref{app:results}. Again, these results show that adaptive partitioning yields significant performance improvements, with speed-ups of up to a factor of three on the larger instances. The computational gains are more modest than for the B-MPCVRP, however, reflecting the different bottleneck in this application: since finding representatives and enumerating duties is relatively cheap, computation times are largely driven by the total number of column generation iterations. Table~\ref{tab:rcsp_root_comparison} confirms that adaptive partitioning is considerably faster per iteration, but requires more iterations to converge (up to 8,658 vs. at most 2,642). This difference can be explained by the fact that the RMP is a highly degenerate, large-scale set covering problem. While the overall effect of adaptive partitioning is already positive, advanced dual stabilization techniques appear necessary to fully leverage its potential for this problem.

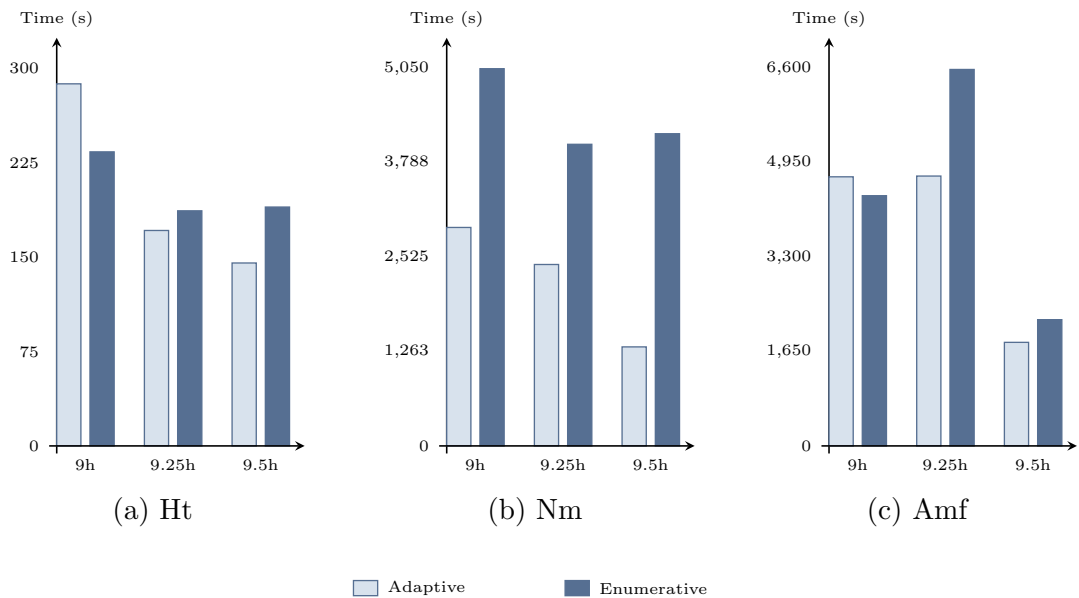
\begin{figure}[h!]
  \centering
  \begin{subfigure}[b]{4.30cm}
    \centering
    \begin{tikzpicture}[x=1cm, y=1cm]

      \node[font=\tiny, black, left] at (-0.1,0) {0};
      \node[font=\tiny, black, left] at (-0.1,1.25) {75};
      \node[font=\tiny, black, left] at (-0.1,2.50) {150};
      \node[font=\tiny, black, left] at (-0.1,3.75) {225};
      \node[font=\tiny, black, left] at (-0.1,5.00) {300};

      \draw[->, >=stealth, black, line width=0.6pt] (0,-0.1) -- (0,5.40)
        node[above, font=\tiny, black] {Time (s)};

      \draw[fill=medBlue!20, draw=darkBlue!75, line width=0.5pt] (0.00, 0) rectangle (0.32, 4.79);
      \draw[fill=darkBlue!75, draw=darkBlue!75, line width=0.5pt] (0.44, 0) rectangle (0.76, 3.89);
      \node[font=\tiny, black, below=1pt] at (0.38, 0) {9h};

      \draw[fill=medBlue!20, draw=darkBlue!75, line width=0.5pt] (1.16, 0) rectangle (1.48, 2.85);
      \draw[fill=darkBlue!75, draw=darkBlue!75, line width=0.5pt] (1.60, 0) rectangle (1.92, 3.11);
      \node[font=\tiny, black, below=1pt] at (1.54, 0) {9.25h};

      \draw[fill=medBlue!20, draw=darkBlue!75, line width=0.5pt] (2.32, 0) rectangle (2.64, 2.42);
      \draw[fill=darkBlue!75, draw=darkBlue!75, line width=0.5pt] (2.76, 0) rectangle (3.08, 3.16);
      \node[font=\tiny, black, below=1pt] at (2.70, 0) {9.5h};

      \draw[->, >=stealth, black, line width=0.6pt] (-0.1,0) -- (3.28,0);
    \end{tikzpicture}
    \caption{Ht}
  \end{subfigure}
  \hspace{0.5cm}
  \begin{subfigure}[b]{4.30cm}
    \centering
    \begin{tikzpicture}[x=1cm, y=1cm]

      \node[font=\tiny, black, left] at (-0.1,0) {0};
      \node[font=\tiny, black, left] at (-0.1,1.25) {1{,}263};
      \node[font=\tiny, black, left] at (-0.1,2.50) {2{,}525};
      \node[font=\tiny, black, left] at (-0.1,3.75) {3{,}788};
      \node[font=\tiny, black, left] at (-0.1,5.00) {5{,}050};

      \draw[->, >=stealth, black, line width=0.6pt] (0,-0.1) -- (0,5.40)
        node[above, font=\tiny, black] {Time (s)};

      \draw[fill=medBlue!20, draw=darkBlue!75, line width=0.5pt] (0.00, 0) rectangle (0.32, 2.89);
      \draw[fill=darkBlue!75, draw=darkBlue!75, line width=0.5pt] (0.44, 0) rectangle (0.76, 4.99);
      \node[font=\tiny, black, below=1pt] at (0.38, 0) {9h};

      \draw[fill=medBlue!20, draw=darkBlue!75, line width=0.5pt] (1.16, 0) rectangle (1.48, 2.40);
      \draw[fill=darkBlue!75, draw=darkBlue!75, line width=0.5pt] (1.60, 0) rectangle (1.92, 3.99);
      \node[font=\tiny, black, below=1pt] at (1.54, 0) {9.25h};

      \draw[fill=medBlue!20, draw=darkBlue!75, line width=0.5pt] (2.32, 0) rectangle (2.64, 1.31);
      \draw[fill=darkBlue!75, draw=darkBlue!75, line width=0.5pt] (2.76, 0) rectangle (3.08, 4.13);
      \node[font=\tiny, black, below=1pt] at (2.70, 0) {9.5h};

      \draw[->, >=stealth, black, line width=0.6pt] (-0.1,0) -- (3.28,0);
    \end{tikzpicture}
    \caption{Nm}
  \end{subfigure}
  \hspace{0.5cm}
  \begin{subfigure}[b]{4.30cm}
    \centering
    \begin{tikzpicture}[x=1cm, y=1cm]

      \node[font=\tiny, black, left] at (-0.1,0) {0};
      \node[font=\tiny, black, left] at (-0.1,1.25) {1{,}650};
      \node[font=\tiny, black, left] at (-0.1,2.50) {3{,}300};
      \node[font=\tiny, black, left] at (-0.1,3.75) {4{,}950};
      \node[font=\tiny, black, left] at (-0.1,5.00) {6{,}600};

      \draw[->, >=stealth, black, line width=0.6pt] (0,-0.1) -- (0,5.40)
        node[above, font=\tiny, black] {Time (s)};

      \draw[fill=medBlue!20, draw=darkBlue!75, line width=0.5pt] (0.00, 0) rectangle (0.32, 3.56);
      \draw[fill=darkBlue!75, draw=darkBlue!75, line width=0.5pt] (0.44, 0) rectangle (0.76, 3.31);
      \node[font=\tiny, black, below=1pt] at (0.38, 0) {9h};

      \draw[fill=medBlue!20, draw=darkBlue!75, line width=0.5pt] (1.16, 0) rectangle (1.48, 3.57);
      \draw[fill=darkBlue!75, draw=darkBlue!75, line width=0.5pt] (1.60, 0) rectangle (1.92, 4.98);
      \node[font=\tiny, black, below=1pt] at (1.54, 0) {9.25h};

      \draw[fill=medBlue!20, draw=darkBlue!75, line width=0.5pt] (2.32, 0) rectangle (2.64, 1.37);
      \draw[fill=darkBlue!75, draw=darkBlue!75, line width=0.5pt] (2.76, 0) rectangle (3.08, 1.67);
      \node[font=\tiny, black, below=1pt] at (2.70, 0) {9.5h};

      \draw[->, >=stealth, black, line width=0.6pt] (-0.1,0) -- (3.28,0);
    \end{tikzpicture}
    \caption{Amf}
  \end{subfigure}
  \par\medskip
  \begin{tikzpicture}[x=1cm, y=1cm]
    \draw[fill=medBlue!20, draw=darkBlue!75, line width=0.5pt] (0,0) rectangle (0.32,0.20);
    \node[font=\tiny, black, right] at (0.32,0.10) {Adaptive};
    \draw[fill=darkBlue!75, draw=darkBlue!75, line width=0.5pt] (2.80,0) rectangle (3.12,0.20);
    \node[font=\tiny, black, right] at (3.12,0.10) {Enumerative};
  \end{tikzpicture}
  \caption{Root node computation time for the RCSP: adaptive partitioning and enumerative benchmark. Results are averaged over all days of the week.}
  \label{fig:rcsp_root_comparison}
\end{figure}

\subsection{Diving Heuristic}
\label{subsec:crew_diving}

Finally, we apply the adaptive partitioning algorithm within a diving heuristic to find primal solutions to the RCSP, following the approach commonly used for large-scale crew scheduling problems \Citep{rahlmann2021robust, breugem2022column, rossum2025benders}. We iteratively solve the RMP using column generation, fix the highest-valued fractional column, and all other fractional columns whose value exceeds 0.6, to one, and repeat until the solution is integral. Table~\ref{tab:rcsp_diving} reports the results, averaged over all days of the week. The heuristic consistently produces high-quality solutions, with an average optimality gap never exceeding 1\% across all crew bases and template lengths. Computation times are reasonable, remaining below two hours for the medium-sized crew bases Ht and Nm, and below four hours for the largest crew base Amf. The results reveal a clear cost reduction as template length increases, most pronounced in the step from 9 to 9.25 hours, consistent with earlier findings \Citep{rahlmann2021robust,rossum2025benders}.

\begin{table}[H]
  \centering
  \caption{Results of diving heuristic with adaptive partitioning for the RCSP. Results are averaged over all days of the week.}
  \label{tab:rcsp_diving}
  \begin{tabular}{ll rrrr}
    \toprule
    Base & Length (h) & Time (s) & LB & UB & Gap (\%) \\
    \midrule
    \multirow{3}{*}{Ht} & 9 & 451 & 10{,}268{,}308 & 10{,}337{,}834 & 0.72 \\
     & 9.25 & 293 & 10{,}102{,}181 & 10{,}193{,}900 & 0.79 \\
     & 9.5 & 460 & 10{,}036{,}793 & 10{,}059{,}791 & 0.19 \\
    \midrule
    \multirow{3}{*}{Nm} & 9 & 4{{,}}118 & 15{,}350{,}868 & 15{,}373{,}500 & 0.14 \\
     & 9.25 & 5{{,}}282 & 15{,}163{,}676 & 15{,}224{,}569 & 0.39 \\
     & 9.5 & 4{{,}}548 & 15{,}051{,}513 & 15{,}223{,}283 & 1.00 \\
    \midrule
    \multirow{3}{*}{Amf} & 9 & 11{{,}}602 & 17{,}433{,}038 & 17{,}506{,}189 & 0.42 \\
     & 9.25 & 9{{,}}055 & 17{,}113{,}477 & 17{,}257{,}383 & 0.82 \\
     & 9.5 & 7{{,}}400 & 17{,}003{,}783 & 17{,}142{,}494 & 0.77 \\
    \bottomrule
  \end{tabular}
\end{table}

\section{Conclusion}
\label{sec:conclusion}

We studied a class of nested path problems in which every column corresponds to a sequence of subpaths satisfying local resource constraints, combined into paths subject to global path resource constraints. We proposed an adaptive partitioning algorithm that solves the pricing problem to optimality in a finite number of iterations. As compared to state-of-the-art pricing algorithms, this approach does not resort to ex-ante enumeration of non-dominated subpaths. Rather, it partitions subpaths into buckets defined by resource intervals, represents each bucket by its subpath of minimum reduced cost, and applies pessimistic and optimistic pricing steps on a coarsened bucket graph. These steps maintain valid upper and lower bounds on the minimum reduced cost, and an adaptive refinement procedure provably closes the gap in a finite number of iterations. Altogether, this paper contributes a novel adaptive decomposition scheme to enhance the pricing algorithm in large-scale column generation and branch-and-price algorithms for nested path problems.

We demonstrated the effectiveness of the algorithm on two applications, with different structural features that uncover different mechanisms underlying the algorithm's benefits. In the balanced multi-period capacitated vehicle routing problem, the adaptive partitioning algorithm yields speed-ups of up to a factor of 13 against an enumerative benchmark, and the resulting branch-price-and-cut algorithm solves three times as many instances to optimality as a subpath-based branch-price-and-cut baseline. In this case, the main bottleneck involves finding representatives within each bucket, so the algorithm's gains stem from avoiding the enumeration of all non-dominated subpaths. In the robust crew scheduling problem, we obtain speed-ups of up to a factor of three and produce primal solutions within 1\% of optimality through a diving heuristic. In this case, finding representative duties is comparatively cheaper, and the computational gains instead come from maintaining a much smaller bucket graph when generating full paths. As a result of these structural differences, the two applications call for different algorithmic configurations. For example, reuse of representatives is essential in the former problem but counterproductive in the latter.

We identify several promising directions for future research. For instance, the refinement procedure could be made more aggressive: in our current implementation, refinement is triggered only when pessimistic pricing fails, but one could apply it earlier when pessimistic pricing still returns columns of negative reduced cost but progress appears to be stalling. Moreover, the algorithm can be applied to other nested path problems, both within and beyond transportation. A challenging test case would be the integrated crew scheduling and rostering problem, a long-standing open problem in railway planning. Tackling this problem would likely require pairing our algorithm with dual stabilization techniques \citep{elhallaoui2005dynamic}.

\bibliographystyle{style/informs2014}
\bibliography{references}

\newpage

\begin{APPENDICES}

\section{Full Results}
\label{app:results}

\begin{table}[h!]
  \centering
  \caption{Root node computational performance of adaptive partitioning per configuration for the B-MPCVRP. Results are averaged over all instances with $n=25$, $T=4$, and $\delta=10\%$.}
  \label{tab:root_strategies_cvrp}
  \begin{tabular}{cccc rrrrr}
    \toprule
    \multicolumn{4}{c}{Parameters} & & \multicolumn{4}{c}{Share of pricing (\%)} \\
    \cmidrule(lr){1-4} \cmidrule(lr){6-9}
    Merge & Midway & Reuse & Width & Time (s) & Fill & Pess. & Opt. & Merge \\
    \midrule
    \multirow{12}{*}{No} & \multirow{6}{*}{Yes} & \multirow{3}{*}{No} & 100 & 174 & 99.6 & 0.3 & 0.1 & 0.0 \\
     &  &  & 250 & 152 & 99.8 & 0.2 & 0.1 & 0.0 \\
     &  &  & 500 & 175 & 99.8 & 0.1 & 0.1 & 0.0 \\
    \cmidrule(lr){3-9}
     &  & \multirow{3}{*}{Yes} & 100 & 135 & 99.3 & 0.6 & 0.1 & 0.0 \\
     &  &  & 250 & 119 & 99.5 & 0.4 & 0.1 & 0.0 \\
     &  &  & 500 & 156 & 99.6 & 0.3 & 0.1 & 0.0 \\
    \cmidrule(lr){2-9}
     & \multirow{6}{*}{No} & \multirow{3}{*}{No} & 100 & 159 & 99.8 & 0.2 & 0.0 & 0.0 \\
     &  &  & 250 & 132 & 99.9 & 0.1 & 0.0 & 0.0 \\
     &  &  & 500 & 270 & 99.9 & 0.1 & 0.0 & 0.0 \\
    \cmidrule(lr){3-9}
     &  & \multirow{3}{*}{Yes} & 100 & 141 & 99.4 & 0.6 & 0.0 & 0.0 \\
     &  &  & 250 & 124 & 99.6 & 0.4 & 0.0 & 0.0 \\
     &  &  & 500 & 216 & 99.6 & 0.4 & 0.0 & 0.0 \\
    \midrule
    \multirow{12}{*}{Yes} & \multirow{6}{*}{Yes} & \multirow{3}{*}{No} & 100 & 126 & 99.4 & 0.2 & 0.1 & 0.3 \\
     &  &  & 250 & 124 & 99.8 & 0.1 & 0.1 & 0.1 \\
     &  &  & 500 & 169 & 99.8 & 0.1 & 0.1 & 0.1 \\
    \cmidrule(lr){3-9}
     &  & \multirow{3}{*}{Yes} & 100 & 109 & 99.1 & 0.7 & 0.1 & 0.2 \\
     &  &  & 250 & 105 & 99.4 & 0.4 & 0.1 & 0.1 \\
     &  &  & 500 & 143 & 99.6 & 0.3 & 0.1 & 0.0 \\
    \cmidrule(lr){2-9}
     & \multirow{6}{*}{No} & \multirow{3}{*}{No} & 100 & 113 & 99.7 & 0.1 & 0.0 & 0.2 \\
     &  &  & 250 & 122 & 99.8 & 0.1 & 0.0 & 0.1 \\
     &  &  & 500 & 242 & 99.8 & 0.1 & 0.0 & 0.1 \\
    \cmidrule(lr){3-9}
     &  & \multirow{3}{*}{Yes} & 100 & 102 & 99.3 & 0.6 & 0.0 & 0.2 \\
     &  &  & 250 & 104 & 99.5 & 0.3 & 0.0 & 0.1 \\
     &  &  & 500 & 197 & 99.6 & 0.4 & 0.0 & 0.1 \\
    \bottomrule
  \end{tabular}
\end{table}

\begin{table}[h!]
  \centering
  \caption{Root node computational performance for the B-MPCVRP: adaptive partitioning and enumerative benchmark. Results are averaged over all five instances and values of $\delta$.}
  \label{tab:cvrp_root_comparison}
  \begin{tabular}{rrrrrrr}
    \toprule
    & & \multicolumn{2}{c}{Adaptive} & \multicolumn{2}{c}{Benchmark} \\
    \cmidrule(lr){3-4} \cmidrule(lr){5-6}
    $n$ & $T$ & Time (s) & Iterations & Time (s) & Iterations \\
    \midrule
    \multirow{3}{*}{15} & 2 & 1.1 & 25.6 & 0.5 & 9.0 \\
     & 3 & 2.2 & 49.8 & 0.9 & 12.8 \\
     & 4 & 3.8 & 115.8 & 1.7 & 23.2 \\
    \midrule
    \multirow{3}{*}{20} & 2 & 4.8 & 66.2 & 9.0 & 12.6 \\
     & 3 & 6.7 & 87.6 & 15.7 & 15.8 \\
     & 4 & 14.9 & 240.0 & 51.9 & 41.4 \\
    \midrule
    \multirow{3}{*}{25} & 2 & 15.8 & 111.4 & 136.5 & 15.2 \\
     & 3 & 33.7 & 180.2 & 202.6 & 13.8 \\
     & 4 & 65.7 & 348.4 & 875.5 & 42.4 \\
    \bottomrule
  \end{tabular}
\end{table}

\begin{table}[h!]
  \centering
  \caption{Root node computational performance of adaptive partitioning per configuration for the RCSP. Results for crew base Ht and a template length of 9 hours, averaged over all days of the week.}
  \label{tab:root_strategies_csp}
  \begin{tabular}{cccc rrrrr}
    \toprule
    \multicolumn{4}{c}{Parameters} & & \multicolumn{4}{c}{Share of pricing (\%)} \\
    \cmidrule(lr){1-4} \cmidrule(lr){6-9}
    Merge & Midway & Reuse & Width (min) & Time (s) & Fill & Pess. & Opt. & Merge \\
    \midrule
    \multirow{12}{*}{No} & \multirow{6}{*}{Yes} & \multirow{3}{*}{No} & 20 & 838 & 6.6 & 85.9 & 7.5 & 0.0 \\
     &  &  & 30 & 483 & 16.4 & 77.1 & 6.6 & 0.0 \\
     &  &  & 40 & 375 & 26.1 & 68.3 & 5.6 & 0.0 \\
    \cmidrule(lr){3-9}
     &  & \multirow{3}{*}{Yes} & 20 & 2{,}402 & 0.6 & 96.7 & 2.7 & 0.0 \\
     &  &  & 30 & 1{,}318 & 1.8 & 97.2 & 1.0 & 0.0 \\
     &  &  & 40 & 1{,}396 & 3.7 & 95.0 & 1.3 & 0.0 \\
    \cmidrule(lr){2-9}
     & \multirow{6}{*}{No} & \multirow{3}{*}{No} & 20 & 871 & 6.6 & 84.8 & 8.6 & 0.0 \\
     &  &  & 30 & 448 & 15.4 & 78.4 & 6.2 & 0.0 \\
     &  &  & 40 & 449 & 27.2 & 68.3 & 4.5 & 0.0 \\
    \cmidrule(lr){3-9}
     &  & \multirow{3}{*}{Yes} & 20 & 2{,}303 & 0.5 & 96.6 & 2.9 & 0.0 \\
     &  &  & 30 & 1{,}428 & 1.8 & 97.4 & 0.9 & 0.0 \\
     &  &  & 40 & 1{,}271 & 3.8 & 95.1 & 1.1 & 0.0 \\
    \midrule
    \multirow{12}{*}{Yes} & \multirow{6}{*}{Yes} & \multirow{3}{*}{No} & 20 & 272 & 12.4 & 20.8 & 2.0 & 64.8 \\
     &  &  & 30 & 232 & 14.2 & 16.4 & 2.6 & 66.7 \\
     &  &  & 40 & 378 & 19.5 & 16.8 & 2.1 & 61.5 \\
    \cmidrule(lr){3-9}
     &  & \multirow{3}{*}{Yes} & 20 & 764 & 1.9 & 91.1 & 0.8 & 6.2 \\
     &  &  & 30 & 535 & 5.2 & 68.5 & 1.0 & 25.4 \\
     &  &  & 40 & 1{,}245 & 6.4 & 65.6 & 0.6 & 27.5 \\
    \cmidrule(lr){2-9}
     & \multirow{6}{*}{No} & \multirow{3}{*}{No} & 20 & 212 & 12.4 & 23.9 & 1.6 & 62.1 \\
     &  &  & 30 & 251 & 17.7 & 16.6 & 1.1 & 64.6 \\
     &  &  & 40 & 385 & 19.9 & 16.4 & 1.0 & 62.7 \\
    \cmidrule(lr){3-9}
     &  & \multirow{3}{*}{Yes} & 20 & 746 & 2.0 & 91.0 & 0.7 & 6.3 \\
     &  &  & 30 & 509 & 5.2 & 71.2 & 0.4 & 23.2 \\
     &  &  & 40 & 1{,}486 & 6.7 & 64.3 & 0.4 & 28.5 \\
    \bottomrule
  \end{tabular}
\end{table}

\begin{table}[H]
  \centering
  \caption{Root node computational performance for the RCSP: adaptive partitioning and enumerative benchmark. Results are averaged over all days of the week.}
  \label{tab:rcsp_root_comparison}
  \begin{tabular}{llrrrr}
    \toprule
    & & \multicolumn{2}{c}{Adaptive} & \multicolumn{2}{c}{Benchmark} \\
    \cmidrule(lr){3-4} \cmidrule(lr){5-6}
    Base & Length & Time (s) & Iterations & Time (s) & Iterations \\
    \midrule
    \multirow{3}{*}{Ht} & 9h & 287 & 1{,}300 & 233 & 289 \\
     & 9.25h & 171 & 467 & 186 & 251 \\
     & 9.5h & 145 & 319 & 190 & 268 \\
    \midrule
    \multirow{3}{*}{Nm} & 9h & 2{,}920 & 4{,}780 & 5{,}035 & 2{,}483 \\
     & 9.25h & 2{,}424 & 4{,}142 & 4{,}028 & 1{,}822 \\
     & 9.5h & 1{,}324 & 2{,}455 & 4{,}173 & 1{,}965 \\
    \midrule
    \multirow{3}{*}{Amf} & 9h & 4{,}699 & 8{,}658 & 4{,}369 & 1{,}917 \\
     & 9.25h & 4{,}710 & 8{,}607 & 6{,}573 & 2{,}642 \\
     & 9.5h & 1{,}807 & 2{,}373 & 2{,}211 & 885 \\
    \bottomrule
  \end{tabular}
\end{table}

\end{APPENDICES}


\end{document}